\documentclass[12pt,a4paper,psamsfonts]{amsart}
\usepackage{amssymb,amscd,amsxtra,calc}
\usepackage{mathrsfs}
\usepackage{cmmib57}
\usepackage{multirow}
\usepackage[all]{xy}
\usepackage[colorlinks,linkcolor=blue,anchorcolor=blue,citecolor=green]{hyperref}
\theoremstyle{plain}
    \newtheorem{thm}{Theorem}[section]

    \newtheorem{corollary}[thm]{Corollary}
    \newtheorem{example}[thm]{Example}
    \newtheorem{lemma}[thm]{Lemma}
    \newtheorem{proposition}[thm]{Proposition}
    \newtheorem{question}[thm]{Question}
    \newtheorem{theorem}[thm]{Theorem}

\theoremstyle{definition}
    \newtheorem{definition}[thm]{Definition}

    \newtheorem{remark}[thm]{Remark}
\theoremstyle{remark}

\newcommand{\Q}{\mathbb{Q}}

\newcommand{\R}{\mathbb{R}}

\newcommand{\Z}{\mathbb{Z}}

\newcommand{\PER}{\operatorname{Per}}

\newcommand{\alb}{\operatorname{alb}}
\newcommand{\Amp}{\operatorname{Amp}}

\newcommand{\Aut}{\operatorname{Aut}}

\newcommand{\diag}{\operatorname{diag}}

\newcommand{\id}{\operatorname{id}}

\newcommand{\NE}{\overline{\operatorname{NE}}}
\newcommand{\Nef}{\operatorname{Nef}}
\newcommand{\NS}{\operatorname{NS}}

\newcommand{\PE}{\operatorname{PE}}
\newcommand{\PEC}{\operatorname{PEC}}

\newcommand{\Sing}{\operatorname{Sing}}

\newcommand{\Supp}{\operatorname{Supp}}

\newcommand{\Codim}{\operatorname{codim}}

\newcommand{\N}{\operatorname{N}}

\newcommand{\Alb}{\operatorname{Alb}}

\newcommand{\Pic}{\operatorname{Pic}}

\newcommand{\alg}{\mathrm{alg}}
\newcommand{\reg}{\mathrm{reg}}

\begin{document}

\title[Building blocks of polarized endomorphisms]
{Building blocks of polarized endomorphisms of normal projective varieties}

\author{Sheng Meng, De-Qi Zhang}

\address
{
\textsc{Department of Mathematics} \endgraf
\textsc{National University of Singapore, 10 Lower Kent Ridge Road,
Singapore 119076
}}
\email{ms@u.nus.edu}
\address
{
\textsc{Department of Mathematics} \endgraf
\textsc{National University of Singapore, 10 Lower Kent Ridge Road,
Singapore 119076
}}
\email{matzdq@nus.edu.sg}

\begin{abstract}
An endomorphism $f$ of a projective variety $X$ is polarized (resp. quasi-polarized) if
$f^\ast H \sim qH$ (linear equivalence) for some ample (resp. nef and big) Cartier divisor $H$ and integer $q > 1$.
First, we use cone analysis to show that a quasi-polarized endomorphism is always polarized,
and the polarized property descends via any equivariant dominant rational map.
Next, we show that a suitable maximal rationally connected fibration (MRC) can be made
$f$-equivariant using a construction of N.~Nakayama,
that $f$ descends to a polarized endomorphism of the base $Y$ of this MRC
and that this $Y$ is a $Q$-abelian variety (quasi-\'etale quotient
of an abelian variety).
Finally, we show that we can run the minimal model program (MMP) $f$-equivariantly for mildly singular $X$
and reach either a $Q$-abelian variety or a Fano variety of Picard number one.

As a consequence, the building blocks of polarized endomorphisms are those of $Q$-abelian varieties and those of Fano varieties of Picard number one.

Along the way, we show that $f$ always descends to a polarized endomorphism of the Albanese variety $\Alb(X)$ of $X$, and that the pullback of a power of $f$ acts as a scalar multiplication on the Neron-Severi group of $X$ (modulo torsion) when $X$ is smooth and rationally connected.

Partial answers about $X$ being of Calabi-Yau type, or Fano type are also given with an extra primitivity
assumption on $f$ which seems necessary by an example.
\end{abstract}



\maketitle
\tableofcontents

\section{Introduction}
%
%
%
%

We work over an algebraically closed field $k$ which has characteristic zero, and is uncountable (only used to guarantee the birational invariance of the rational connectedness property).
Let $f$ be a surjective endomorphism of a projective variety $X$. We say
that $f$ is {\it polarized} (resp. {\it quasi-polarized}), if there is an ample (resp. nef and big) Cartier divisor $H$ such that $f^{\ast}H \sim qH$ (linear equivalence) for some integer $q>1$.
If $X$ is a point, then the only trivial endomorphism is polarized by convention.


Let $X$ be a projective variety of dimension $n$. We refer to Definition \ref{def-equ-car} for the numerical equivalence ($\equiv$) of $\R$-Cartier divisors and Definition \ref{def-num-cyc} for the weak numerical equivalence ($\equiv_w$) of $r$-cycles with real coefficients.
Denote by $\N^1(X):=\NS(X) \otimes_{\Z} \mathbb{R}$ for the N\'eron-Severi group $\NS(X)$.
One can also regard $\N^1(X)$ as the quotient vector space of $\R$-Cartier divisors modulo the numerical equivalence; see Definition \ref{def-equ-car}.
Denote by $\N_r(X)$ the quotient vector space of $r$-cycles modulo the weak numerical equivalence.

Suppose further $X$ is normal.
Then the numerical equivalence and the weak numerical equivalence are the same for $\R$-Cartier divisors;
in particular,
the natural map $\N^1(X) \to \N_{n-1}(X)$ is well defined and an injection (cf. Definition \ref{def-num-cyc} and Lemma \ref{nrn1}).
A Weil $\mathbb{R}$-divisor $F$ is said to be {\it big} if $F=A+E$ for some ample $\mathbb{Q}$-Cartier divisor $A\in \N^1(X)$ and pseudo-effective Weil $\mathbb{R}$-divisor $E$; see Definition \ref{def-cones}.

A surjective endomorphism $f:X\to X$ of a projective variety $X$ is a finite morphism.
In fact, $f$ induces an automorphism $f^{\ast}:\N^1(X)\to \N^1(X)$.
So an ample divisor is the pull back of some divisor, which, together with the projection formula, imply the finiteness of $f$.
Suppose further $f^{\ast}H\sim qH$ for some nef and big Cartier divisor $H$ and $q>0$,
then, by taking top self-intersection, the projection formula implies
the relation between $\deg f$ and $q$: $\,\, \deg f=q^{\dim(X)}.$

Now we state our main results.

\begin{proposition}\label{big-polar-THM}(cf.~Proposition \ref{big-polar})  Let $f:X\to X$ be a surjective endomorphism of a projective variety $X$ and $q>0$ a rational number.
Assume one of the following two conditions.
\begin{itemize}
\item[(1)] $f^\ast H \equiv qH$
for some big $\mathbb{R}$-Cartier divisor $H$.
\item[(2)] $X$ is normal and $f^\ast H\equiv_w qH$ for some big Weil $\mathbb{R}$-divisor $H$.
\end{itemize}
Then $q$ is an integer and $f^\ast A\equiv qA$ for some ample Cartier
divisor $A$.
Further, if $q>1$, then $f$ is polarized.
In particular, quasi-polarized endomorphisms are polarized.
\end{proposition}

Given a normal projective variety $X$, denote by $\Aut(X)$ the full automorphism group of $X$ and $\Aut_0(X)$ its neutral connected component.
Let $B$ be a Cartier divisor.
Denote by $\Aut_{[B]}(X):=\{g\in \Aut(X)\,|\, g^\ast B\equiv B\}$.
If $X$ is smooth and $B$ is ample, then $[\Aut_{[B]}(X):\Aut_0(X)]<\infty$ by \cite[Proposition 2.2]{Li}.
Generally, let $G$ be a subgroup of $\Aut(X)$, such that for any $g\in G$, $g^\ast B_g \equiv B_g$ for some big Cartier divisor $B_g$.
Then $[G:G\cap \Aut_0(X)]<\infty]$ by \cite[Theorem 2.1]{DHZ}.
Now applying Proposition \ref{big-polar-THM}, we can further generalize this result to the following.

\begin{theorem}(cf.~Theorem \ref{aut}) Let $X$ be a normal projective variety. Let $G$ be a subgroup of $\Aut(X)$, such that for any $g\in G$, $g^\ast B_g \equiv_w B_g$ for some big Weil $\mathbb{R}$-divisor $B_g$.  Then $[G:G\cap \Aut_0(X)]<\infty$.
\end{theorem}

The polarized property descends via any equivariant dominant rational map. Indeed, we prove:

\begin{theorem}\label{descend-thm}(cf.~Theorem \ref{desends-polar}) Let $\pi:X\dashrightarrow Y$ be a dominant rational map between two projective varieties and let $f:X\to X$ and $g:Y\to Y$ be two surjective endomorphisms such that $g\circ\pi=\pi\circ f$.
Suppose $f$ is polarized. Then $g$ is polarized; and  $(\deg f)^{\dim(Y)}=(\deg g)^{\dim(X)}$.
\end{theorem}

Given a projective variety $X$, pick any smooth model $p:X'\to X$, we define the Albanese map $\alb_{X}$ of $X$
as $\alb_{X'}\circ p^{-1}$:
$$X \overset{p^{-1}}{\dashrightarrow} X' \overset{\alb_{X'}}{\longrightarrow} \Alb(X') =: \Alb(X) .$$
Clearly, $\alb_X$ and $\Alb(X)$
are independent of the choice of $X'$. By the universal property of the Albanese map,
any surjective endomorphism (or even dominant rational self-map) $f$ of $X$ descends to a surjective endomorphism $f|_{\Alb(X)}$ of $\Alb(X)$.
The following Corollary \ref{mainthm} is an application of Theorem \ref{descend-thm}.

\begin{corollary}\label{mainthm} Let $X$ be a projective variety with a polarized endomorphism $f:X\to X$ and
let $\alb_X:X\dashrightarrow \Alb(X)$ be the Albanese map of $X$.
Then the following are true.
\begin{itemize}
\item[(1)] $\alb_X$ is a dominant rational map.
\item[(2)]The endomorphism $f|_{\Alb(X)}$ of $\Alb(X)$ induced from $f$ is polarized.
\end{itemize}
\end{corollary}

\begin{remark}\label{main-que-rmk} Corollary \ref{mainthm} affirmatively answers Krieger - Reschke \cite[Question 1.10]{Kr}.
\end{remark}

We refer to \cite[Chapters 2 and 5]{KM} for the definitions and the properties of log canonical (lc), Kawamata log terminal (klt), canonical and terminal singularities.
A normal projective variety $Y$ is $Q$-{\it abelian} if there exists a finite surjective morphism $A\to Y$ \'etale in codimension one
(or {\it quasi-\'etale} in short) with $A$ an abelian variety. By the ramification divisor formula, $K_Y\sim_{\mathbb{Q}} 0$.

Let $V$ be a projective variety of dimension $n$. $V$ is said to be {\it uniruled} if there is a dominant rational map $\mathbb{P}^1\times U\dashrightarrow X$ with $\dim(U)=n-1$.
$V$ is said to be {\it rationally connected}, in the sense of Campana and Kollar-Miyaoka-Mori,
(\cite{Ca}, \cite{KMM}),
if any two points of $V$ are connected by a rational curve, which is equivalent to saying that two general points of $V$ are connected by a rational curve since our ground field is uncountable (see \cite{Ko}).
Given a uniruled projective variety $X$,
there is a fibration: $\pi:X\dashrightarrow Y$, such that $Y$ is a non-uniruled normal projective variety
(cf.~\cite{GHS}) and the graph of $\pi$ over $Y$ has the general fibre rationally connected.
We call it an {\it MRC} (maximal rationally connected) fibration in the sense of
Campana and Kollar-Miyaoka-Mori and this fibration is unique up to birational equivalence (cf.~ \cite{Ko}).
The Albanese map of $X$ always factors through the MRC fibration; see Lemma \ref{albmrc}.

However, in general, fixing one MRC fibration,
a surjective endomorphism of $X$ descends only to a dominant rational self-map of $Y$.
Nevertheless, we have the next result.

\begin{proposition}\label{PropA} Let $X$ be a normal projective variety with a polarized endomorphism $f:X\to X$.
Then there is a special MRC fibration $\pi:X\dashrightarrow Y$ in the sense of Nakayama \cite{Na10}
(which is the identity map when $X$ is non-uniruled)
together with a (well-defined) surjective endomorphism $g$ of $Y$, such that the following are true.
\begin{itemize}
\item[(1)]
$g\circ \pi=\pi\circ f$; $g$ is polarized.
\item[(2)]
$Y$ is $Q$-abelian (with only canonical singularities). Hence there is a finite surjective morphism $T \to Y$ \'etale in codimension one with $T$ an abelian variety
and $g$ lifts to a polarized endomorphism $g_T$ of $T$.
\item[(3)]
Let $\bar{\Gamma}_{X/Y}$ be the normalization of the graph of $\pi$. Then the induced morphism $\bar{\Gamma}_{X/Y}\to Y$ is equi-dimensional
with each fibre (irreducible) rationally connected.
\item[(4)]
If $X$ has only klt singularities, then $\pi$ is a morphism.
\end{itemize}
\end{proposition}

\begin{remark}\label{main-rmk-dense}
(1) By N.~Fakhruddin (cf.~\cite{Fa}), the set of $g$-periodic points $\PER(Y,g):=\{y\in Y\,|\,g^s(y)=y\  \text{for some}\  s>0\}$ is Zariski dense in $Y$.
Thus the fibre $\bar{\Gamma}_y$ of the normalization of the graph of $\pi$ over each $y\in \PER(Y,g)$ is rationally connected and admits a polarized endomorphism.

(2) By virtue of Proposition \ref{PropA}, the building blocks of polarized endomorphisms are those on $Q$-abelian varieties and those on rationally connected varieties.
\end{remark}

In view of Remark \ref{main-rmk-dense}, we still need to consider an arbitrary polarized endomorphism $f$  of
a rationally connected variety $X$.
If $X$ is mildly singular, we can run
the minimal model program (MMP) and reach a Fano fibration provided that $K_X$ is not pseudo-effective, and continue the MMP for the
base of the Fano fibration and so on.
Now the problem is that we need to guarantee the equivariance of the MMP with respect to
$f$ (or its positive power), i.e., to make sure that the extremal rays contracted during the MMP are fixed by $f$ or its positive power.
This is not easy, because there might be infinitely many extremal rays, being divisorial type, or flip type, or Fano type.

In Theorem \ref{scalarthm},
we manage to show the equivariance of the MMP with respect to a positive power of $f$, generalizing
results in \cite{Zh-comp} for lower dimensions to all dimensions.
This is done in Section \ref{MMP}.

For a surjective endomorphism $f:X\to X$ of a projective variety $X$,
we say that $f^\ast|_{\N^1(X)}$ is a {\it scalar multiplication} if there is some $q\in \R$ such that $f^*x=qx$ for all $x\in \N^1(X)$.
When $\alb_X$ is a surjective morphism,
if $f^\ast|_{\N^1(X)}$ is a scalar multiplication, then so is $(f|_{\Alb(X)})^{\ast}|_{\N^1(\Alb(X))}$.
The converse may not hold even if we assume $f$ to be polarized and replace $f$ by any positive power; see Example \ref{ex2}. Nevertheless, we have Theorem \ref{scalarthm} (4) below.

\begin{theorem}\label{scalarthm}
Let $f:X\to X$ be a polarized endomorphism of a $\mathbb{Q}$-factorial klt projective variety $X$.
Then, replacing $f$ by a positive power, there exist a $Q$-abelian variety $Y$, a morphism $X\to Y$, and an $f$-equivariant relative MMP over $Y$ $$X=X_1\dashrightarrow \cdots \dashrightarrow X_i \dashrightarrow \cdots \dashrightarrow X_r=Y$$ (i.e. $f=f_1$ descends to $f_i$ on each $X_i$), with every $X_i \dashrightarrow X_{i+1}$ a divisorial contraction, a flip or a Fano contraction, of a $K_{X_i}$-negative extremal ray, such that we have:
\begin{itemize}
\item[(1)]
If $K_X$ is pseudo-effective, then $X=Y$ and it is $Q$-abelian
(see Proposition  \ref{PropA} or Lemma \ref{struccor} for the lifting of $f$).
\item[(2)]
If $K_X$ is not pseudo-effective, then for each $i$, $X_i\to Y$ is equi-dimensional holomorphic with every fibre (irreducible) rationally connected and $f_i$ is polarized by some ample Cartier divisor $H_i$.
The $X_{r-1}\to X_r = Y$ is a Fano contraction.
\item[(3)]
$\N^1(X)$ is spanned by the pullbacks of $\N^1(Y)$ and those $\{H_i\}_{i<r}$
which are $f^\ast$-eigenvectors corresponding to the same eigenvalue $q=(\deg f_i)^{1/\dim(X_i)}$ (independent of $i$).
\item[(4)]
$f^*|_{\N^1(X)}$ is a scalar multiplication: $f^*|_{\N^1(X)} = q \, \id$, if and only if so  is $f_{r}^\ast$.
\end{itemize}
\end{theorem}

\begin{remark}In Theorem \ref{scalarthm}, if we weaken the klt singularities assumption on $X$ to lc singularities, then our lemmas assert that
we still can run $f$-equivariant MMP. We refer to \cite{Fu15} for the LMMP of $\Q$-factorial log canonical pair. However, we could not show that this MMP terminates and could not claim assertions (1) - (2) in Theorem \ref{scalarthm}.
\end{remark}

%

A normal projective variety $X$ is of {\it Calabi-Yau type} (resp. {\it Fano type}), if there is a boundary $\Q$-divisor
$\Delta \ge 0$, such that the pair $(X, \Delta)$ has at worst lc (resp.~klt) singularities and $K_X+\Delta\sim_{\mathbb{Q}}0$ (resp.~ $-(K_X+\Delta)$ is ample).
Applying Theorem \ref{scalarthm} and working a bit more, we have the following result.

\begin{theorem}\label{thm-smooth-rc}  Let $f:X\to X$ be a polarized endomorphism of a smooth rationally connected projective variety $X$.
Then, replacing $f$ by a positive power, we have:
\begin{itemize}
\item[(1)] $f^\ast|_{\N^1(X)}$ is a scalar multiplication. Namely, $f^\ast|_{\N^1(X)} = q \, \id$ for some $q > 1$.
\item[(2)] The number $s$ of all prime divisors $D_i$
with either $f^{-1}(D_i) = D_i$ or $\kappa(X, D_i) = 0$, satisfies $s \le \dim (X) + \rho(X)$, where $\rho(X)$ is the Picard number of $X$.
\item[(3)] The Iitaka $D$-dimensions satisfy $\kappa(X, -K_X) \ge \kappa(X, -(K_X + \sum\limits_{i=1}^s D_i)) \ge 0$.
\item[(4)] If (i) $s = \dim (X) + \rho(X)$, or (ii) $\kappa(X, -(K_X + \sum\limits_{i=1}^s D_i)) = 0$ or (iii) $D_1$ is non-uniruled,
then $K_X + \sum\limits_{i=1}^s D_i \sim_\mathbb{Q} 0$ and $X$ is of Calabi-Yau type (with $s = 1$ in Case (iii)).
\end{itemize}
\end{theorem}

Our Theorem \ref{thm-smooth-rc} (3) is slightly stronger than that in \cite[Theorem C]{Bo}, where they proved that $-K_X$ is pseudo-effective, but they did not assume $X$ is rationally connected.

Let $X$ be a normal projective variety.
A polarized endomorphism $f:X\to X$ is {\it imprimitive} if there is a dominant rational map $\pi:X \dashrightarrow Y$ to a normal projective variety $Y$,
with $0<\dim(Y)<\dim(X)$, such that $f$ descends to a polarized endomorphism $f_Y$ of $Y$ (i.e.~$f_Y\circ\pi=\pi\circ f$). We say that $f$ is {\it primitive} if it is not imprimitive.


\begin{corollary}\label{main-cor-k} Let $f:X\to X$ be a polarized endomorphism of a $\mathbb{Q}$-factorial klt projective variety $X$ such that:
\begin{itemize}
\item[(i)]
$X$ is not a $Q$-abelian variety, and
\item[(ii)]
$f^s$ is primitive for all $s > 0$.
\end{itemize}
Then, replacing $f$ by a positive power, we have the following assertions.
\begin{itemize}
\item[(1)]$(X, f)$ is equivariantly birational to a Fano variety $X_{r-1}$ of Picard number one, and $f^\ast|_{\N^1(X)}$ is a scalar multiplication.
Precisely, running the MMP on $X$ we get an $f$-equivariant sequence $X=X_1 \dashrightarrow   \cdots \dashrightarrow X_{r-1}$ of divisorial contractions and flips, such that $X_{r-1}$ is a Fano variety of Picard number one with a polarized endomorphism  $f_{r-1}$ of $X_{r-1}$ (induced from $f$).
\item[(2)] $-K_X$ is big.
\end{itemize}
\end{corollary}

\begin{remark}
$ $
\begin{itemize}
\item[(1)] Corollary \ref{main-cor-k} partially answers \cite[Question 1.5]{yZh} about $X$ being of Calabi-Yau type. It also
answers \cite[Question 1.6 (2)]{yZh} about $X$ being of Fano type (but only up to $f$-equivariant birational map)
with an extra primitivity assumption on the pair $(X, f)$ which or something similar is needed, otherwise, Example  \ref{ex-xu}
gives a negative answer in the general case.
\item[(2)] By Remark \ref{main-rmk-dense} and Corollary \ref{main-cor-k},
we may say the building blocks of polarized endomorphisms are those on ($Q$-) abelian varieties and Fano varieties of Picard number one. This belief was stated and confirmed in \cite{Zh-comp} in dimension $\le 4$.
\end{itemize}
\end{remark}

The following question is natural from Theorem \ref{thm-smooth-rc} without assuming $X$ to be smooth.

\begin{question}\label{que-terminal-scalar} Let $f:X\to X$ be a polarized endomorphism of a rationally connected normal projective variety $X$.
Assume that $X$ has at worst $\mathbb{Q}$-factorial terminal singularities.
Is $(f^s)^\ast|_{\N^1(X)}$ a scalar multiplication for some $s>0$?
\end{question}

The above question has a positive answer when $\dim(X) \le 3$; see \cite[Theorem 1.2]{Zh-comp}.
Without the rational connectedness condition,
Question \ref{que-terminal-scalar} has a negative answer (see Example \ref{ex1}).
In view of Example \ref{ex2}, though it is a $K3$ surface with canonical singularities,
the terminality condition might be needed too.




%
%
%
%
\section{Preliminary results}

Let $X$ be a projective variety.
We use Cartier divisor $H$ (a Cartier divisor is integral, unless otherwise indicated)
and its corresponding invertible sheaf $\mathcal{O}(H)$ interchangeably.
Denote by $\Pic(X)$ the group of Cartier divisors modulo linear equivalence and $\Pic^0(X)$ the subgroup of the classes in $\Pic(X)$ which are algebraically equivalent to $0$.
Denote by $\N^1(X):=\NS(X) \otimes_{\Z} \mathbb{R}$ for the N\'eron-Severi group $\NS(X)$, where $\NS(X)=\Pic(X)/\Pic^0(X)$.

\begin{definition}[equivalences of Cartier divisors]\label{def-equ-car}
Let $X$ be a projective variety of dimension $n$ and $D$ a Cartier divisor.
$D$ is said to be {\it $\tau$-equivalent} to $0$ if $mD$ is algebraically equivalent to $0$ for some integer $m\neq 0$.
$D$ is said to be {\it numerically equivalent} to $0$ if $D\cdot C=0$ for any curve $C$ of $X$.
We extend the above two equivalences to the case of $\R$-Cartier divisors.
Let $E$ be an $\R$-Cartier divisor.
$E$ is said to be $\tau$-equivalent to $0$ (denoted by $E\sim_\tau 0$) if $E=\sum_i a_i E_i$ where $a_i\in \R$ and $E_i$ is a Cartier divisor $\tau$-equivalent to $0$.
Similarly, $E$ is said to be numerically equivalent to $0$ (denoted by $E\equiv 0$) if $E=\sum_i a_i E_i$ where $a_i\in \R$ and $E_i$ is a Cartier divisor numerically equivalent to $0$.
It is known that $E\sim_\tau 0$ if and only if $E\equiv 0$ (cf.~\cite[Theorem 9.6.3]{FGI}).
Therefore, one can also regard $\N^1(X)$ as the quotient vector space of $\R$-Cartier divisors modulo the numerical equivalence.
\end{definition}

\begin{definition}[weak numerical equivalence of cycles]\label{def-num-cyc}
Let $X$ be a projective variety of dimension $n$.
Let $Z$ and $Z'$ be two $r$-cycles with real coefficients.
$Z$ is said to be {\it weakly numerically equivalent} to $Z'$ (denoted by $Z\equiv_w Z'$) if $(Z-Z')\cdot H_1\cdots H_r=0$ for any Cartier divisors $H_1,\cdots, H_r$.
Denote by $\N_r(X)$ the quotient vector space of $r$-cycles with real coefficients modulo the weak numerical equivalence.

Suppose further $X$ is normal.
$\N_{n-1}(X)$ is then the quotient vector space of Weil $\R$-divisors modulo the weak numerical equivalence.
It is known that Cartier divisors of $X$ are Weil divisors.
Let $D$ be an $\R$-Cartier divisor.
If $D\equiv 0$, then $D\equiv_w 0$.
The converse is true by the lemma below.
Therefore, one can regard $\N^1(X)$ as a subspace of $\N_{n-1}(X)$ and they are the same if $X$ is also $\mathbb{Q}$-factorial.
\end{definition}

\begin{lemma}\label{nrn1}(cf.~\cite[Lemma 3.2]{Zh-tams}) Let $X$ be a projective variety of dimension $n$.
Let $H_1,\cdots,H_{n-1}$ be ample $\mathbb{R}$-Cartier divisors and $M$ an $\mathbb{R}$-Cartier divisor.
Suppose that $$H_1\cdots H_{n-1}\cdot M = 0 =H_1\cdots H_{n-2}\cdot M^2.$$ Then $M\equiv 0$ (i.e., $M\cdot C = 0$ for any curve $C$ of $X$).
In particular, if $X$ is normal, then $\N^1(X)$ is a subspace of $\N_{n-1}(X)$.
\end{lemma}

\begin{definition}\label{def-cones} Let $X$ be a projective variety of dimension $n$. We define:
\begin{itemize}
\item $\Amp(X)$, the set of classes of ample $\mathbb{R}$-Cartier divisors in $\N^1(X)$.
\item $\Nef(X)$, the set of classes of nef $\mathbb{R}$-Cartier divisors in $\N^1(X)$.
\item $\PEC(X)$, the closure of the set of classes of effective $\mathbb{R}$-Cartier divisors in $\N^1(X)$.
\item $\PE(X)$, the closure of the set of classes of effective $(n-1)$-cycles with $\R$-coefficients in $\N_{n-1}(X)$.
\end{itemize}
Clearly, the above sets are all convex cones and contain no lines.
Note that $\PEC(X)$ can also be regarded as the closure of the set of classes of big $\R$-Cartier divisors in $\N^1(X)$.
Let $f:X\to X$ be a finite surjective endomorphism.
We may define pullback of cycles for $f$, such that $f^\ast$ induces an automorphism of $\N_r(X)$ and $f_\ast f^\ast=(\deg f)\,\id$; see \cite[Section 2]{Zh-comp}.
Note that the above cones are $(f^\ast)^{\pm 1}$-invariant.

Suppose further $X$ is normal.
Let $D$ be a Weil $\R$-divisor.
Then $D$ is said to be {\it pseudo-effective} if its class $[D]\in \PE(X)$.
In addition, $D$ is said to be {\it big} if $D=A+E$ for some ample $\mathbb{Q}$-Cartier
divisor and pseudo-effective Weil $\mathbb{R}$-divisor $E$, see \cite[Theorem 3.5]{FKL} for equivalent definitions.
Note that $\PE(X)$ can also be regarded as the closure of the set of classes of big Weil $\R$-divisors in $\N_{n-1}(X)$.
\end{definition}

\begin{definition}\label{def-q} Let $X$ be a normal projective variety.
Define:
\begin{itemize}
\item[(1)]
$q(X)=h^1(X,\mathcal{O}_X)=\dim H^1(X,\mathcal{O}_X)$ (the irregularity);
\item[(2)]
$\tilde{q}(X)=q(\tilde{X})$ with $\tilde{X}$ a smooth projective model of $X$; and
\item[(3)]
$q^\natural(X)=\sup\{\tilde{q}(X')\,|\,X'\to X \text{ is finite surjective and \'etale in codimension one}\}.$
\end{itemize}
\end{definition}

\begin{definition} Let $V$ be a finite dimensional real normed vector space and $S\subseteq V$ a subset.

The convex hull generated by $S$ is defined as
$$\{\sum\limits_{i\in I}a_i x_i \, |\, a_i\ge 0, x_i\in S, \sum\limits_{i\in I}a_i=1, |I|<\infty\}.$$
Suppose $S$ is bounded, i.e., there exists some $N>0$, such that $|s|<N$ for any $s\in S$.
Then $|\sum\limits_{i\in I}a_i x_i|\le \sum\limits_{i\in I}a_i \max\limits_{i\in I}\{|x_i|\} =
\max\limits_{i\in I}\{|x_i|\}<N$.  So the closure of the convex hull generated by $S$ is bounded.
Let $m$ be the dimension of the vector space spanned by $S$.
If $S$ is a finite set, then the convex hull generated by $S$ is a polytope which is covered by finitely many $m$-simplexes.
So for arbitrary $S$ (possibly an infinite set) and the convex hull $D$ generated by $S$, we may write the element $d\in D$ as $d=\sum\limits_{i= 1}^{m+1} a_i x_i$ with $a_i\ge 0$, $x_i\in S$ and $\sum\limits_{i= 1}^{m+1} a_i=1$.

The convex cone generated by $S$ is defined as
$$\{\sum\limits_{i\in I}a_i x_i \, |\, a_i\ge 0, x_i\in S, |I|<\infty\}.$$

Let $A$ and $B$ be two subsets of $V$. We define $d(A,B)=\inf\{|a-b|\, :\, a\in A, b\in B\}$.
Denote by $B(x,r):=\{x'\in V\, :\, |x-x'|<r\}$.

Let $C$ be a convex cone in $V$. Let $W$ be the subspace of $V$ spanned by $C$.
Denote by $\overline{C}$ the usual closure of $C$ in $V$.
Note that $\overline{C}\subseteq W$.
Denote by $C^{\circ}$ the interior part of $C$ as the biggest open subset of $W$ contained in $C$.
Denote by $\partial C:=\overline{C}-C^{\circ}$ the boundary of $C$.
Note that $C^\circ=\overline{C}^\circ$ and $\partial C=\partial \overline{C}$.

Let $C_1\subseteq C_2$ be two convex cones of $V$.
We say $C_1$ is an {\it extremal face} of $C_2$ if $u+v\in C_1$ implies that $u, v\in C_1$ for any $u, v\in C_2$.
An extremal face of a closed convex cone is always closed.
\end{definition}

\begin{lemma}\label{face} Let $V$ be a positive dimensional real normed vector space.
Let $C\subseteq V$ be a closed convex cone which spans $V$ and contains no line.
Let $C'\subseteq \partial C$ be a convex subcone.
Then $C'$ is contained in a closed extremal face $F\subseteq \partial C$. In particular, $C'$ is contained in a unique minimal closed extremal face in $\partial C$.
\end{lemma}
\begin{proof} Let $F:=\{x\in C\, |\, x+y\in C' \text{ for some } y\in C\}$.

If $x_1, x_2\in F$, then $x_1+y_1\in C'$ and $x_2+y_2\in C'$ for some $y_1, y_2\in C$.
For any $a\ge 0$, $a x_1+a y_1\in C'$ and $x_1+x_2+y_1+y_2\in C'$. So $F$ is a convex cone.

If $x\in F\cap C^\circ$, then $x+y\in C'$ for some $y\in C$ and $B(x,r)\subseteq C^\circ$ for some $r>0$. So $x+y\in B(x+y,r)=y+B(x,r)\subseteq C^\circ$ and hence $C'\cap C^\circ \neq \emptyset$, a contradiction. So $F\subseteq \partial C$.

If $x\in C'$, then $x+0=x\in C'$. So $C'\subseteq F$.

If $p, q\in C$ and $p+q\in F$, then $p+q+s\in C'$ for some $s\in C$.
By the construction of $F$, we have both $p, q\in F$. So $F$ is an extremal face.

Since $C$ is closed, $F$ is closed too.

By taking the intersection of all extremal faces containing $C'$, we get the minimal one.
In fact, $F$ is already the minimal one. Suppose $F'$ is an extremal face containing $C'$. Then for any $x\in F$, $x+y\in C'\subseteq F'$ for some $y\in C$. So $x\in F'$.
\end{proof}

\begin{lemma}\label{dis} Let $V$ be a positive dimensional real normed vector space.
Let $C\subseteq V$ be a closed convex cone containing no lines and $F\subsetneq C$ a positive dimensional  extremal face.
Fix $d>0$ and $k>0$. Let $S$ be the set of all $x\in C$ with $d(x,F)\ge d$ and $|x|\le k$.
Let $B$ be the closure of the convex cone generated by $S$. Then $d(F,B)>0$.
\end{lemma}
\begin{proof} Let $B'$ be the convex cone generated by $S$.  Let $n$ be the dimension of the vector space spanned by $S$.
Define $$\eta(p, x_1, \cdots, x_{n+1}):=d(p, D_{x_1, \cdots, x_{n+1}}),$$ where $p\in F$, $x_i\in S$ and $D_{x_1, \cdots, x_{n+1}}$ is the convex polytope generated by $x_1, \cdots, x_{n+1}$.
Clearly, $\eta$ is a continuous function from $F\times S^{\times (n+1)}$ to $\mathbb{R}_{\ge0}$. If $\eta(p, x_1, \cdots, x_{n+1})=0$, then $p\in D_{x_1, \cdots, x_{n+1}}$. Since $F$ is an extremal face, $x_i\in F$, a contradiction. So $\eta>0$.
Let $F_{>2k}:=\{p\in F\, :\, |p|>2k\}$ and $F_{\le 2k}=F-F_{>2k}$. Then $A:=F_{\le 2k}\times S^{\times (n+1)}$ is compact and hence $\eta|_A\ge d_1$ for some $d_1>0$.
Note that $D(x_1,\cdots, x_{n+1})$ is bounded by $k$. So for $A':=F_{> 2k}\times S^{\times (n+1)}$, $\eta|_{A'}> k$.
Since $B'=\cup_{x_1,\cdots, x_{n+1}\in S} D_{x_1, \cdots, x_{n+1}}$, $d(F, B')=\inf\limits_{\substack{p\in F; \, x_1,\cdots, x_{n+1}\in S}} \{d(p, D_{x_1, \cdots, x_{n+1}})\}\ge \min\{d_1, k\}>0$ and hence $d(F, B)>0$.
\end{proof}

For a linear map $f: V \to V$ of a finite dimensional real normed vector space $V$,
denote by $||f||$ the {\it norm} of $f$.

\begin{proposition}\label{cone} Let $f:V\to V$ be an invertible linear map of a positive dimensional real normed vector space $V$ such that $f^{\pm1}(C)=C$ for a closed convex cone $C\subseteq V$ which spans $V$ and contains no line.
Let $q$ be a positive number. Then (1) and (2) below are equivalent.
\begin{itemize}
\item[(1)] $f(x)=q x$ for some $x\in C^\circ$ (the interior part of $C$).
\item[(2)] There exists a constant $N>0$, such that $\frac{||f^i||}{q^i}< N$ for any $i\in \mathbb{Z}$.
\end{itemize}

Let $W\subseteq V$ be the eigenspace of $f$ corresponding to the eigenvalue $q$.
If (1) or (2) above is true, then $f$ is a diagonalizable linear map with all eigenvalues of modulus $q$.
So $F_\infty:=\lim\limits_{n\to +\infty}\frac{1}{n}\sum\limits_{i=0}^{n-1}\frac{f^i}{q^i}$ is a well-defined linear map onto $W$ and $F_\infty|_W= \id_W$.
\end{proposition}
\begin{proof} We may assume $q=1$ after replacing $f$ by $f/q$.

$(1) \Rightarrow (2)$
We may assume $|x|=1$.
Since $x\in C^\circ$ and $C$ spans $V$, there exists some $t>0$, such that $x+tv\in C$ for any $v\in V$ with $|v|\le 1$.
Since $C$ contains no line, there exists some $s>0$, such that $x-x'\not\in C$ for any $x'\in C$ with $|x'|\ge s$.
Suppose that $||f^i||$ is not bounded.
Since $C$ spans $V$, there exist some $i_0\in \mathbb{Z}$ and $y\in C$ with $|y|=1$, such that $|tf^{i_0}(y)|>s$.
Since $x-ty\in C$, $f^{i_0}(x-ty)\in C$. However, $f^{i_0}(x-ty)=x-tf^{i_0}(y)\not\in C$, a contradiction.

Suppose $(2)$ is true. If either the spectral radius of $f$ is greater than $1$ or $f$ has a nontrivial Jordan block whose eigenvalue is of modulus $1$.
Then $\lim\limits_{n\to +\infty}||f^n||=+\infty$, a contradiction.
Similarly, the spectral radius of $f^{-1}$ is $1$.
Therefore, the last assertion of the proposition follows.

$(2) \Rightarrow (1)$ Set $n=\dim(V)$ and $m=\dim(W)$.
If $m=n$, it is trivial.
Suppose $m<n$. If $W\cap C^\circ \neq \emptyset$, then we are done. Suppose $W\cap C^\circ =\emptyset$. Since $f^{\pm 1}(C)=C$, $F_\infty(C)\subseteq C$ and hence $F_\infty(C)\subseteq W\cap C$. Note that $W$ is spanned by $F_\infty (C)$.
Then $W\cap C$ is an $m$-dimensional closed convex subcone of $C$ in $\partial C$.
By Lemma \ref{face}, $W\cap C\subseteq F$ for some minimal closed extremal face $F$ of $C$ in $\partial C$.
Note that $f^{\pm 1}(F)$ are still minimal closed extremal faces  in $\partial C$ containing $W\cap C$.
By the uniqueness, $f^{\pm1}(F)=F$.
Fix any $y\in C-F$ and $d:=d(y,F)>0$.
Then by $(2)$, for any $z\in F$, $d\le |y-f^{-i}(z)|=|f^{-i}(f^i(y)-z)|< N|f^i(y)-z|$.
So $d(f^i(y),F)\ge\frac{d}{N}$ for any $i\in \Z$.
Let $B$ be the closure of the convex hull generated by all $f^i(y)$ ($i \in \Z$).
By (2), $B$ is bounded, $f^{\pm1}(B)=B$ and $d(B,F)>0$ by Lemma \ref{dis}. Since $C\cap W\subseteq F$ and $B\subseteq C$, $B\cap W= B\cap C\cap W\subseteq B\cap F=\emptyset$. However, by Brouwer-fixed point theorem, $f(b)=b$ for some $b\in B$. So we get a contradiction.
\end{proof}

\begin{lemma}\label{kmlem}(cf.~\cite[Lemma 2.62]{KM}) Let $f:X\to Y$ be a birational morphism of normal projective varieties.
Assume that $Y$ is $\mathbb{Q}$-factorial.
Then there is an effective $f$-exceptional divisor $F$ such that $-F$ is $f$-ample.
\end{lemma}
\begin{proposition}\label{kmprop}(cf.~\cite[Proposition 1.45]{KM}) Let $f:X\to Y$ be a morphism of projective varieties with $M$ an ample divisor on $Y$.
If $L$ is an $f$-ample Cartier divisor on $Y$, then $L+\nu f^{\ast}M$ is ample for $\nu\gg 1$.
\end{proposition}

The following lemma slightly extends \cite[Lemma 2.2]{Na-Zh}.
\begin{lemma}\label{nzlem} Let $X$ be a projective variety of dimension $n$ and $D$ an $\R$-Cartier divisor.
Assume the following two conditions:

(1) $D\cdot G\cdot L_1\cdots L_{n-2}\ge 0$ for any effective Cartier divisor $G$ and any $L_i\in\Nef(X)$.

(2) $D\cdot H_1\cdots H_{n-1}=0$ for some nef and big $\mathbb{R}$-Cartier divisors $H_1,\cdots,H_{n-1}$.
\\
Then $D\equiv 0$.
\end{lemma}
\begin{proof} By the proof of \cite[Lemma 2.2]{Na-Zh}, $D\cdot A^{n-1}=0$ and $D^2\cdot A^{n-2}=0$ for some ample divisor $A$. So $D\equiv 0$ by Lemma \ref{nrn1}.
\end{proof}

\begin{lemma}\label{ext}(cf.~\cite[Lemma 2.11]{Zh-comp})
Let $f : X \to X$ be a surjective endomorphism of a normal projective variety $X$ with $K_X$ being $\Q$-Cartier.
Let $\NE(X)$ be the closure of the convex cone generated by classes of effective $1$-cycles in $\N_1(X)$.
Let $R_C := \R_{\ge 0}[C] \subseteq \NE(X)$ be an extremal ray generated by a curve $C$ (not necessarily $K_X$-negative).
Then we have:
\begin{itemize}
\item[(1)]
$R_{f(C)}$ is an extremal ray.
\item[(2)]
If $C_1$ is another curve such that $f(C_1) = C$, then $R_{C_1}$ is an extremal ray.
\item[(3)]
Denote by $\Sigma_C$ the set of curves whose classes are in $R_C$.
Then $f(\Sigma_C) = \Sigma_{f(C)}$.
\item[(4)]
If $R_{C_1}$ is extremal, then $\Sigma_{C_1} = f^{-1}(\Sigma_{f(C_1)})
:= \{D \, | \, f(D) \in \Sigma_{f(C_1)}\}$.
\end{itemize}
\end{lemma}

Let $f:X\to Y$ be a surjective morphism between normal projective varieties.
Then $f$ has connected fibres if and only if $f_\ast\mathcal{O}_{X}=\mathcal{O}_Y$ (cf.~\cite[Chapter III, Corollary 11.3 and Corollary 11.5]{Har}) and hence the composition of two such morphisms still has connected fibres.
In particular, the general fibre of $f$ is connected if and only if all the fibres are connected and hence a birational morphism between normal projective varieties always has connected fibres (cf.~\cite[Chapter III, Corollary 11.4]{Har}).
Suppose that $f$ has connected fibres.
Let $X'$ be a resolution of $X$ with $f':X'\to Y$ the induced morphism. Then $f'$ has connected fibres and the general fibre of $f'$ is smooth by generic smoothness (cf.~\cite[Chapter III, Corollary 10.7]{Har}).
In particular, the general fibre of $f'$ is irreducible.
Note that each fibre of $f$ is an image of $f'$.
So the general fibre of $f$ is irreducible.

\begin{definition}\label{def-gen-rc}
Let $f:X\dashrightarrow Y$ be a dominant map between two normal projective varieties and $f_{\bar{\Gamma}}:\bar{\Gamma}\to Y$ the induced morphism with $\bar{\Gamma}$ the normalization of the graph of $f$.
We say that $f$ has {\it the general fibre rationally connected} if the general fibre of $f_{\bar{\Gamma}}$ is rationally connected.

Let $f_1:X_1\to Y_1$ be a surjective morphism birationally equivalent to $f$ with $X_1$ and $Y_1$ being normal projective.
Then $f$ has the general fibre rationally connected if and only if so does $f_1$.
This is because rational connectedness is a birational property (cf.~\cite[Chapter IV, Proposition 3.6]{Ko}).
\end{definition}

\begin{lemma}\label{rc+rc} Let $f:X\dashrightarrow Y$ and $g:Y\dashrightarrow Z$ be two dominant maps between normal projective varieties.
Suppose that $f$ and $g$ have the general fibre rationally connected.
Then $g\circ f$ has the general fibre rationally connected.
\end{lemma}

\begin{proof} We may assume $f$ and $g$ are surjective morphisms between smooth projective varieties.
Let $U$ be an open dense subset of $Y$ such that each fibre of $f$ over $U$ is rationally connected.
Let $V$ be an open dense subset of $Z$ such that $g\circ f$ is smooth over $V$ and each fibre of $g$ over $V$ is rationally connected and has a nonempty intersection with $U$.
Then for any $z\in V$, we have a natural surjective morphism $f_z:X_z\to Y_z$, where $X_z=(g\circ f)^{-1}(z)$ while $Y_z=g^{-1}(z)$ is rationally connected.
Since $Y_z\cap U\neq\emptyset$, $f_z$ has the general fibre rationally connected.
By \cite[Corollary 1.3]{GHS}, $X_z$ is rationally connected.
So $g\circ f$ has the general fibre rationally connected.
\end{proof}

\begin{lemma}\label{mmp-rc} Let $X$ be a $\Q$-factorial klt normal projective variety.
Suppose that $X=X_1\dashrightarrow\cdots \dashrightarrow X_r=Y$ is a sequence of divisorial contractions, flips or Fano contractions, of $K_{X_i}$-negative extremal rays.
Let $f:X\dashrightarrow Y$ be the composition.
Then $f$ has the general fibre rationally connected.
\end{lemma}
\begin{proof} By Lemma \ref{rc+rc}, it suffices to consider the case when $f$ is a Fano contraction.
Then the general fibre of $f$ is a klt Fano variety and hence rationally connected by \cite{qZh}.
\end{proof}

\section{Properties of (quasi-) polarized endomorphisms}

\begin{lemma}\label{basic-deg}Let $f:X\to X$ be a surjective endomorphism of a projective variety $X$, such that
$f^{\ast}H\equiv qH$ for some nef and big $\mathbb{R}$-Cartier divisor $H$ and $q>0$.
Then $\deg f=q^{\dim(X)}.$
\end{lemma}
\begin{proof}
We may assume $n := \dim(X) > 0$.
By the projection formula, $(f^\ast H)^n=(\deg f) H^n=q^nH^n$.
Since $H^n>0$, $\deg f=q^n$.
\end{proof}

\begin{lemma}\label{abelian-deg} Let $f:A \to A$ be a surjective endomorphism of an abelian variety $A$.
Let $Z$ be a subvariety of $A$ such that $f^{-1}(Z) = Z$.
Then $\deg f|_Z = \deg f$.
\end{lemma}
\begin{proof} $f$ is \'etale by the ramification divisor formula and the purity of branch loci.
Then $\deg f|_Z=|f^{-1}(z)|=\deg f$ for any $z\in Z$.
\end{proof}

\begin{lemma}\label{sub-deg} Let $f:X\to X$ be a polarized endomorphism of a projective variety $X$ with $\deg f=q^{\dim(X)}$ and let $Z$ be a closed subvariety of $X$ with $f(Z)=Z$.
Then $\deg f|_Z=q^{\dim(Z)}$.
\end{lemma}
\begin{proof} We may assume $f^\ast H\sim qH$ for some ample divisor $H$ of $X$.
Then $H|_Z$ is also an ample divisor of $Z$ and $(f|_Z)^\ast(H|_Z)\sim q(H|_Z)$.
So $\deg f|_Z=q^{\dim(Z)}$ by Lemma \ref{basic-deg}.
\end{proof}

%

The following lemma is true for not-necessarily normal projective varieties by using the same proof of \cite[Lemma 2.3]{Na-Zh}.
\begin{lemma}\label{NZ-linear} (cf.~\cite[Lemma 2.3]{Na-Zh}) Let $f:X\to X$ be a surjective endomorphism of a projective variety $X$ such that $f^\ast H\equiv qH$ for some integer $q>1$ and ample Cartier divisor $H$. Then there is an ample Cartier divisor $H'\equiv H$, such that $f^\ast H'\sim qH'$. In particular, $f$ is polarized.
\end{lemma}

\begin{lemma}\label{R-int} Let $f:X\to X$ be a surjective endomorphism of a projective variety $X$.
Suppose that there is an ample $\mathbb{R}$-Cartier divisor $H$ such that $f^{\ast}H\equiv qH$ for some rational number $q>0$.
Then $q$ is an integer and $f^{\ast}H'\equiv qH'$ for some ample Cartier (integral) divisor $H'$.
\end{lemma}

\begin{proof} By Lemma \ref{basic-deg}, $q^{\dim(X)}=\deg f$. So $q$ is an algebraic integer and also rational by assumption. Hence, $q$ is an integer.

Let $W\subseteq \N^1(X)$ be the eigenspace of $f^{\ast}|_{\N^1(X)}$ with eigenvalue $q$ and $W_{\mathbb{Q}}$ the set of all $\mathbb{Q}$-Cartier divisor classes in $W$.
Note that $f^{\ast}|_{\N^1(X)}$ is determined by $f^{\ast}|_{\NS_{\mathbb{Q}}(X)}$.
So $W_{\mathbb{Q}}$ is dense in $W$.
Set $F_\infty:=\lim\limits_{n\to +\infty}\frac{1}{n}\sum\limits_{i=0}^{n-1}\frac{(f^i)^\ast|_{\N^1(X)}}{q^i}$.
By Proposition \ref{cone}, $F_\infty$ is a well-defined projection from $\N^1(X)$ onto $W$.
Note that $\N^1(X)$ is spanned by $\Amp(X)$, $F_\infty(W)=W$, and $F_\infty$ is an open map from $\N^1(X)$ onto $W$ by Proposition \ref{cone}.
Then $F_\infty(H+\Amp(X))\cap W_{\mathbb{Q}}\neq \emptyset$.
In particular, $H':=F_\infty(H+D)$ is an ample $\mathbb{Q}$-Cartier divisor for some $D\in \Amp(X)$ and $f^{\ast}[H']= q[H']$ in $\N^1(X)$.
Replacing $H'$ by a multiple, we are done.
\end{proof}

\begin{proposition}\label{big-polar} Let $f:X\to X$ be a surjective endomorphism of a projective variety $X$ and $q>0$ a rational number.
Suppose one of the following is true.
\begin{itemize}
\item[(1)] $f^\ast H \equiv qH$ for some big $\mathbb{R}$-Cartier divisor $H$.
\item[(2)] $X$ is normal and $f^\ast H\equiv_w qH$ for some big Weil $\mathbb{R}$-divisor $H$.
\end{itemize}
Then $q$ is an integer and $f^\ast A\equiv qA$ for some ample Cartier divisor $A$.
Further, if $q>1$, then $f$ is polarized.
In particular, quasi-polarized endomorphisms are polarized.
\end{proposition}
\begin{proof}
Set $n := \dim(X)$. Clearly, $f^\ast$ induces automorphisms on $\N_{n-1}(X)$ and $\N^1(X)$.

For (1), applying Proposition \ref{cone} to $f^\ast|_{\N^1(X)}$, $\PEC(X)$ and its interior point $H$, Proposition \ref{cone}(2) holds for $f^\ast|_{\N^1(X)}$.

For (2), applying Proposition \ref{cone} to $f^\ast|_{\N_{n-1}(X)}$, $\PE(X)$ and its interior point $H$, Proposition \ref{cone}(2) holds for $f^\ast|_{\N_{n-1}(X)}$ and hence for $f^\ast|_{\N^1(X)}$, since $\N^1(X)$ is a subspace of $\N_{n-1}(X)$ by Lemma \ref{nrn1}.

Now for both (1) and (2), applying Proposition \ref{cone} to $f^\ast|_{\N^1(X)}$ and $\Nef(X)$, $f^\ast|_{\N^1(X)}$ has an eigenvector in $\Nef(X)^\circ$.
So $f^\ast A\equiv qA$ for some ample $\mathbb{R}$-Cartier divisor $A$.
By Lemma \ref{R-int}, $q$ is an integer and we may assume $A$ is Cartier.

If $q>1$, then there exists an ample Cartier divisor $A'\equiv A$, such that $f^\ast A'\sim qA'$ by Lemma \ref{NZ-linear}.
\end{proof}

Let $X$ be a projective variety. Denote by $\Aut(X)$ the full automorphism group of $X$ and $\Aut_0(X)$ the neutral connected component of $\Aut(X)$.
\begin{theorem}\label{aut} Let $X$ be a normal projective variety. Let $G$ be a subgroup of $\Aut(X)$, such that for any $g\in G$, $g^\ast B_g\equiv_w B_g$ for some big Weil $\mathbb{R}$-divisor $B_g$.  Then $[G:G\cap \Aut_0(X)]<\infty$.
\end{theorem}
\begin{proof} Take an $\Aut(X)$-equivariant projective resolution $\pi:X'\to X$.
We may regard $G$ and $\Aut_0(X)$ as subgroups of $\Aut(X')$.
By \cite[Proposition 2.1]{Br11}, we can identify $\Aut_0(X')$ with $\Aut_0(X)$.  For any $g\in G$, $g$ fixes some ample class $A_X$ in $\Amp(X)$ by Proposition \ref{big-polar}. Since $\pi$ is birational, $A_{X'}=\pi^\ast A_X$ is big and $g$ fixes $A_{X'}$. By \cite[Theorem 2.1]{DHZ}, $[G:G\cap \Aut_0(X)]=[G:G\cap \Aut_0(X')]<\infty$.
\end{proof}

\begin{lemma}\label{lem-lift-graph} Let $\pi:X\dashrightarrow Y$ be a dominant rational map between two projective varieties and let $f:X\to X$ and $g:Y\to Y$
be two surjective endomorphisms such that $g\circ\pi=\pi\circ f$.
Let $W$ be the graph of $\pi$ and denote by
$p_1:W\to X$ and $p_2:W\to Y$ the two natural projections.
Then there is a surjective endomorphism $h:W\to W$ such that $p_1\circ h=f\circ p_1$ and $p_2\circ h=g\circ p_2$.
Furthermore, if $f$ is polarized, then $h$ is polarized.
\end{lemma}
\begin{proof} Note that $f$ and $g$ induce a natural surjective endomorphism $f\times g:X\times Y\to X\times Y$ via $(x,y)\mapsto (f(x),g(y))$.
Set $h=(f\times g)|_W$.
We check that $h(W)=W$.
Let $U$ be an open dense subset of $X$ such that $\pi|_{U}$ and $\pi|_{f(U)}$ are both well-defined morphisms.
Then $p_1$ is isomorphic over $U$ and $f(U)$.
For any $(x,y)\in p_1^{-1}(U)$, $y=\pi(x)$ and $h(x,y)=(f(x),g(y))=(f(x),g(\pi(x)))=(f(x),\pi(f(x)))\in p_1^{-1}(f(U))\subseteq W$.
Then $h(p_1^{-1}(U))\subseteq W$.
Note that $p_1^{-1}(U)$ is open dense in $W$, $W$ is a projective subvariety of $X\times Y$, and $f\times g$ is finite surjective.
So $h(W)=W$.
Since $p_1\circ h (x,y)=f(x)=f\circ p_1(x,y)$ for any $(x,y)\in p_1^{-1}(U)$, we have $p_1\circ h=f\circ p_1$.
Similarly, $p_2\circ h=g\circ p_2$.

Suppose $f^*H\sim qH$ for some ample divisor $H$ of $X$ and integer $q>1$.
Let $H'=p_1^*H$.
Then $h^*H'\sim qH'$.
Since $p_1$ is birational, $H'$ is nef and big.
So $h$ is quasi-polarized and hence polarized by Proposition \ref{big-polar}.
\end{proof}

\begin{lemma}\label{lem-lift-normal} Let $f:X\to X$
be a surjective endomorphism of a projective variety $X$.
Let $\pi:X_1\to X$ be the normalization of $X$.
Then there is a surjective endomorphism $f_1:X_1\to X_1$ such that $\pi\circ f_1=f\circ \pi$.
Furthermore, if $f$ is polarized, then $f_1$ is polarized.
\end{lemma}
\begin{proof} By the universal property of normalization, such $f_1$ exists.
Suppose $f^*H\sim qH$ for some ample divisor $H$ of $X$ and integer $q>1$.
Let $H_1=\pi^*H$.
Then $H_1$ is ample and $f_1^*H_1\sim qH_1$.
\end{proof}


\begin{lemma}\label{basic-deg2} Let $\pi:X\dashrightarrow Y$ be a dominant rational map between two projective varieties and let $f:X\to X$ and $g:Y\to Y$
be two polarized endomorphisms such that $g\circ\pi=\pi\circ f$.
Then the eigenvalues of $f^\ast|_{\N^1(X)}$ are of modulus $q$ if and only if so are the eigenvalues of $g^\ast|_{\N^1(Y)}$ (if $Y$ is a point, we assume this is always true).
In particular, $(\deg f)^{\dim(Y)}=(\deg g)^{\dim(X)}$.
\end{lemma}

\begin{proof} Replacing $X$ by the graph of $\pi$ and by Lemma \ref{lem-lift-graph}, we may assume $\pi$ is a surjective morphism.
Set $m=\dim (X)$ and $n=\dim (Y)$.
Suppose that $f^\ast H_X=pH_X$ and $g^\ast H_Y\sim qH_Y$ for some ample divisors $H_X\in \N^1(X)$, $H_Y\in \N^1(Y)$ and $p, q>1$.
Since $\pi$ is surjective, $\pi^\ast: \N^1(Y)\to \N^1(X)$ is an injection.
If $Y$ is not a point, then $\N^1(Y)$ is of positive dimension.
Applying Proposition \ref{cone} to the cones $\Nef(X)$ and $\Nef(Y)$, we have $p=q$.
The last assertion then follows from Lemma \ref{basic-deg}.
\end{proof}

\begin{theorem}\label{desends-polar} Let $\pi:X\dashrightarrow Y$ be a dominant rational map between two projective varieties and let $f:X\to X$ and $g:Y\to Y$ be two surjective endomorphisms such that $g\circ\pi=\pi\circ f$. Suppose $f$ is polarized. Then $g$ is polarized; and  $(\deg f)^{\dim(Y)}=(\deg g)^{\dim(X)}$.
\end{theorem}

\begin{proof} Replacing $X$ by the graph of $\pi$ and by Lemma \ref{lem-lift-graph}, we may assume $\pi$ is a surjective morphism.
We may also assume $\dim(Y)>0$.
Since $f$ is polarized, $f^\ast|_{\N^1(X)}$ satisfies Proposition \ref{cone}(2). Since $\pi$ is surjective, $\N^1(Y)$ can be viewed as a subspace of $\N^1(X)$ and hence $g^\ast|_{\N^1(Y)}$ also satisfies Proposition \ref{cone}(2);
so the eigenvalues of $g^\ast|_{\N^1(Y)}$ are of modulus greater than $1$.
Therefore, Proposition \ref{cone}(1) and Lemma \ref{NZ-linear} imply that $g$ is polarized. The last formula follows from Lemma \ref{basic-deg2}.
\end{proof}

\begin{corollary}\label{finite-polar} Let $\pi:X\dashrightarrow Y$ be a generically finite dominant rational map between two projective varieties and let $f:X\to X$ and $g:Y\to Y$ be two surjective endomorphisms such that $g\circ\pi=\pi\circ f$.
Then $f$ is polarized if and only if so is $g$.
\end{corollary}
\begin{proof} Replacing $X$ by the graph of $\pi$ and by Lemma \ref{lem-lift-graph}, we may assume $\pi$ is a generically finite surjective morphism. If $f$ is polarized, then $g$ is polarized by Theorem \ref{desends-polar}.
Suppose $g^*H_Y\sim qH_Y$ for some ample divisor $H_Y$ and $q>1$.
Let $H_X=\pi^*H_Y$.
Since $\pi$ is generically finite surjective, $H_X$ is nef and big.
Note that $f^*H_X\sim qH_X$.
So $f$ is quasi-polarized and hence polarized by Proposition \ref{big-polar}.
\end{proof}

%

%
%
%
%
\section{Special MRC fibration and the non-uniruled case}

We refer to \cite[Section 4]{Na10} for the following result.
\begin{lemma}\label{smrc}(cf.~\cite[Theorem 4.19]{Na10}) Let $f:X\to X$ be a surjective endomorphism of a normal
projective variety.
Let $\pi: X\dashrightarrow Y$ be the special MRC fibration in the sense of \cite[before Theorem 4.18]{Na10}
with $Y$ non-uniruled (cf.~\cite{GHS}).
Then there is a surjective endomorphism $h: Y \to Y$ such that $\pi\circ f=h\circ \pi$.
\end{lemma}

\begin{lemma}\label{albmrc} Let $\pi: X\dashrightarrow Y$ be a dominant rational map between two normal projective varieties.
Then we have the following commutative diagram:
$$\xymatrix{
X\ar@{.>}[r]^\pi\ar@{.>}[d]^{\alb_X}&Y\ar@{.>}[d]^{\alb_Y}\\
\Alb(X)\ar[r]^p&\Alb(Y)
}$$
where $\alb_X$ and $\alb_Y$ are the Albanses maps of $X$ and $Y$ respectively.
Suppose $\pi$  has the general fibre rationally connected (cf.~Definition \ref{def-gen-rc}).
Then $p$ is an isomorphism.
\end{lemma}

\begin{proof} The existence of $p$ follows from the universal property of the Albanese map.
For convenience, we may replace $X$ and $Y$ by suitable smooth models such that $\pi$, $\alb_X$ and $\alb_Y$ are well-defined morphisms.
Note that $\pi$ has the general fibre rationally connected and every map from a rationally connected variety to an abelian variety is trivial.
So there is a Zariski open dense subset $U$ of $Y$ such that $\alb_X(\pi^{-1}(y))$ is a point for any $y\in U$.
By \cite[Lemma 14]{Ka81}, there is a rational map $s:Y\dashrightarrow \Alb(X)$, such that $s\circ \pi=\alb_X$.
Since $\pi$ is surjective, $p\circ s=\alb_Y$.
On the other hand, by the universal property of the Albanese map, $s=t\circ \alb_Y$ for some morphism $t:\Alb(Y)\to \Alb(X)$.
Note that $(p\circ t)\circ\alb_Y=\alb_Y$ and $(t\circ p)\circ \alb_X=\alb_X$.
Then $p\circ t$ and $t\circ p$ are both the identity maps, by the universal property of albansese maps.
Hence, $p$ is an isomorphism.
\end{proof}



\begin{lemma}\label{struclem}  Let $f:X\to X$ be a polarized endomorphism of a non-uniruled normal projective variety $X$.
Then $X$ is $Q$-abelian with canonical singularities and $K_X\sim_\Q 0$.
Further, there is an abelian variety $A$, such that the following diagram is commutative:
$$\xymatrix{
A \ar[r]^{\tau} \ar[d]_{f_A}& X\ar[d]_{f}\\
A \ar[r]^{\tau} & X }$$
where $f_A:A\to A$ is a polarized endomorphism and $\tau$ is a finite surjective morphism which is \'etale in codimension one.
\end{lemma}

\begin{proof} By \cite[Theorem 1.21]{GKP} and \cite[Theorem 3.2]{Na-Zh}, $X$ is $Q$-abelian with only canonical singularities.

By \cite[Proposition 3.5]{Na-Zh}, there exist an abelian variety $A$ and a weak Calabi-Yau variety $S$, such that the following diagram is commutative:
$$\xymatrix{
A\times S \ar[r]^{\tau} \ar[d]_{f_A\times f_S}& X\ar[d]_{f}\\
A\times S \ar[r]^{\tau} & X }$$
where $f_A:A\to A$, $f_S:S\to S$ are polarized endomorphisms, and $\tau$ is a finite surjective morphism which is \'etale in codimension one.
Since $S$ is non-uniruled, $S$ is $Q$-abelian by \cite[Theorem 1.21]{GKP} and hence $\dim(S)=q^\circ(S)=0$ in the notation of \cite{Na-Zh} as in their definition of weak Calabi-Yau.
\end{proof}

\begin{remark}\label{rmk-abelian-ample}
Assume $X$ is a $Q$-abelian variety or just assume
there is a finite surjective morphism $A\to X$ with $A$ an abelian variety.
Since $A$ is a homogeneous variety, any effective divisor on $A$ is nef.
The same holds on $X$ by the projection formula.
Hence if $f:X\to X$ is quasi-polarized, it is also
polarized without using Proposition \ref{big-polar}.
\end{remark}

We refer to the proof of \cite[Proposition 3.5]{Na-Zh} for the following lemma.

\begin{lemma}\label{index1} Let $X$ be a normal projective variety with klt singularities and $K_X\sim_{\mathbb{Q}}0$
and let $\sigma:\hat{X}\to X$ be the global index-one cover.
Then for any surjective endomorphism $f:X\to X$,
there is a surjective endomorphism $\hat{f}:\hat{X}\to\hat{X}$ such that $\sigma\circ\hat{f}=f\circ\sigma$.
\end{lemma}

\begin{lemma}\label{struccor} Let $X$ be a normal projective variety with klt singularities and $K_X\sim_{\mathbb{Q}}0$ (this is the case when $X$ is $Q$-abelian) and let $f:X\to X$ be a polarized endomorphism.
Then there exist a finite surjective morphism
$\tau:A\to X$ \'etale in codimension one with $A$ an abelian variety and a polarized endomorphism $f_A:A\to A$, such that $\tau\circ f_A=f\circ\tau$.
In particular, $X$ is $Q$-abelian.
\end{lemma}

\begin{proof} 
Let $\sigma:\hat{X}\to X$ be the global index-one cover, i.e.~ the minimal quasi-\'etale cyclic covering satisfying $K_{\hat{X}}\sim 0$.
Then there is a polarized endomorphism $\hat{f}:\hat{X}\to\hat{X}$ satisfying $\sigma\circ\hat{f}=f\circ\sigma$ by Lemmas \ref{index1} and Corollary \ref{finite-polar}.
Since $K_{\hat{X}}$ is Cartier and $\hat{X}$ is klt, $\hat{X}$ has canonical singularities.
In particular, $\kappa(\hat{X})=0$ and hence $\hat{X}$ is non-uniruled.
Now by Lemma \ref{struclem}, there exist a finite surjective morphism  $\tau':A\to \hat{X}$ which is \'etale in codimension one with $A$ an abelian variety and a polarized endomorphism $f_A:A\to A$,
such that $\tau'\circ f_A=\hat{f}\circ\tau'$.
Thus the lemma is proved since $\tau=\sigma\circ\tau'$ is still \'etale in codimension one by the ramification divisor formula.
\end{proof}

\begin{lemma}\label{qabelian-nosub} Let $X$ be a $Q$-abelian variety and $f:X\to X$ a surjective endomorphism. Assume the existence of
a non-empty closed subset $Z\subsetneq X$ and $s>0$, such that $f^{-s}(Z)=Z$. Then $f$ is not polarized.
\end{lemma}
\begin{proof} Replacing $f$ by $f^s$, we may assume $f^{-1}(Z)=Z$.
Suppose that $f$ is polarized.
By Lemma \ref{struccor}, there exist a finite surjective morphism $\tau:A\to X$ with $A$ an abelian variety and a polarized endomorphism $f_A:A\to A$, such that $\tau\circ f_A=f\circ\tau$.
Clearly, $f_A^{-1}(\tau^{-1}(Z))=\tau^{-1}(Z)$.
So we may assume that $X$ is an abelian variety; and replacing $f$ by a positive power, we may also assume that $Z$ is irreducible.
By Lemma \ref{abelian-deg}, $\deg f|_Z=\deg f$; and by Lemma \ref{sub-deg}, $\deg f|_Z=(\deg f)^{\dim(Z)/\dim(X)}$.
Since $\dim(Z)<\dim(X)$ and $\deg f>1$ by Lemma \ref{basic-deg}, we get a contradiction.
\end{proof}

\section{Proof of Corollary \ref{mainthm} and Proposition \ref{PropA}}

We begin with the following lemmas.

\begin{lemma}\label{alb-morphism}(cf.~\cite[Proposition 2.3]{Re} or ~\cite[Lemma 8.1]{Ka}) Let $X$ be a normal projective variety having only rational singularities (i.e. there exists a resolution $f:Y\to X$ such that $R^if_{\ast}\mathcal{O}_Y=0$ for $i>0$). Then $f^{\ast}:\Pic^0(X)\to \Pic^0(Y)$ is an isomorphism, and $\alb_X$ is a morphism. In particular, if $h^1(X,\mathcal{O}_X) \neq 0$, then $\alb_X$ is nontrivial.
\end{lemma}


\begin{lemma}\label{fibres-rc+irr} Let $\pi:X\to Y$ be a surjective morphism between normal projective varieties with connected fibres.
Let $f:X\to X$ and $g:Y\to Y$ be two polarized endomorphisms such that $g\circ\pi=\pi\circ f$.
Suppose that $Y$ is $Q$-abelian.
Then the following are true.
\begin{itemize}
\item[(1)]
All the fibres of $\pi$ are irreducible.
\item[(2)]
$\pi$ is equi-dimensional.
\item[(3)]
If the general fibre of $\pi$ is rationally connected, then all the fibres of $\pi$ are rationally connected.
\end{itemize}
\end{lemma}
\begin{proof} First we claim that $f(\pi^{-1}(y))=\pi^{-1}(g(y))$ for any $y\in Y$.
Suppose there is a closed point $y$ of $Y$ such that $f|_{\pi^{-1}(y)}:\pi^{-1}(y)\to \pi^{-1}(g(y))$ is not surjective.
Let $S=g^{-1}(g(y))-\{y\}$.
Then $S\neq \emptyset$ and $U:=X-\pi^{-1}(S)$ is an open dense subset of $X$.
Since $f$ is an open map, $f(U)$ is an open dense subset of $X$. In particular, $f(U)\cap \pi^{-1}(g(y))$ is open in $\pi^{-1}(g(y))$. Note that $f(U)=(X-\pi^{-1}(g(y)))\cup f(\pi^{-1}(y))$. So $f(U)\cap \pi^{-1}(g(y))=f(\pi^{-1}(y))$ is open in $\pi^{-1}(g(y))$. Since $f$ is also a closed map, the set $f(\pi^{-1}(y))$ is both open and closed in the connected fibre $\pi^{-1}(g(y))$ and hence $f(\pi^{-1}(y))=\pi^{-1}(g(y))$.
So the claim is proved.

Let  $$\Sigma_1:=\{y\in Y\,|\, \pi^{-1}(y) \text{ is not irreducible}\}.$$
Note that $f(\pi^{-1}(y))=\pi^{-1}(g(y))$.
Then $g^{-1}(\Sigma_1)\subseteq \Sigma_1$ and hence  $g^{-1}(\overline{\Sigma_1})\subseteq \overline{\Sigma_1}$.
Since $\overline{\Sigma_1}$ is closed and has finitely many irreducible components, $g^{-1}(\overline{\Sigma_1})=\overline{\Sigma_1}$.
By Lemma \ref{qabelian-nosub}, $\Sigma_1=\emptyset$.
So (1) is proved.

Let $$\Sigma_2:=\{y\in Y\,|\,\dim(\pi^{-1}(y))>\dim(X)-\dim(Y)\},$$ and
$$\Sigma_3:=\{y\in Y\,|\, \pi^{-1}(y) \text{ is not rationally connected}\}.$$

By (1), $\pi$ is equi-dimensional outside $\Sigma_2$.
Since $f$ is finite surjective, $g^{-1}(\Sigma_2)\subseteq \Sigma_2$.
By (1), all the fibres of $\pi$ outside $\Sigma_3$ are rationally connected.
Note that the image of a rationally connected variety is rationally connected.
So $g^{-1}(\Sigma_3)\subseteq \Sigma_3$.
Now the same reason above implies that $\Sigma_2=\emptyset$.
Similarly, $\Sigma_3=\emptyset$ if the general fibre of $\pi$ is rationally connected.
\end{proof}


\begin{lemma}\label{mor-q-abelian} Let $\pi:X\dashrightarrow Y$ be a dominant rational map between normal projective varieties.
Suppose that $(X,\Delta)$ is a klt pair for some effective $\mathbb{Q}$-divisor $\Delta$ and $Y$ is $Q$-abelian.
Suppose further that the normalization of the graph $\Gamma_{X/Y}$ is equi-dimensional over $Y$
(this holds when $k(Y)$ is algebraically closed in $k(X)$, $f: X \to X$ is polarized and
$f$ descends to some polarized $f_Y : Y \to Y$; see Lemma \ref{fibres-rc+irr}).
Then $\pi$ is a morphism.
\end{lemma}

\begin{proof} Let $W$ be the normalization of the graph $\Gamma_{X/Y}$ and $p_1:W\to X$ and $p_2:W\to Y$ the two projections.
Let $\tau_1:A\to Y$ be a finite surjective morphism \'etale in codimension one with $A$ an abelian variety.
Let $W'$ be be an irreducible component of the normalization of $W\times_Y A$ which dominates $W$ and $\tau_2:W'\to W$ and $p_2':W'\to A$ the two projections.
Taking the Stein factorization of the composition $W'\to W\to X$, we get a birational morphism $p_1':W'\to X'$ and a finite morphism $\tau_3:X'\to X$.
$$\xymatrix{
X'\ar[d]^{\tau_3}         & W'\ar[l]_{p_1'}\ar[d]^{\tau_2}\ar[r]^{p_2'} &A\ar[d]^{\tau_1}\\
X                & W\ar[l]_{p_1}\ar[r]^{p_2}     &Y}$$

Since $p_2$ is equi-dimensional, by the base change, $\tau_2$ is \'etale in codimension one.
Let $U\subseteq X$ be the domain of $p_1^{-1}:X\dashrightarrow W$.
Then, $\Codim(X-U) \ge 2$, and the restriction
$\tau_3^{-1}(U)\to U$ of $\tau_3$ is \'etale in codimension one, since so is $\tau_2$.
Therefore, $\tau_3$ is \'etale in codimension one. In particular, by the ramification divisor formula, $K_{X'}+\Delta'=\tau_3^\ast(K_X+\Delta)$ with $\Delta'=\tau_3^\ast \Delta$ an effective $\Q$-divisor.
Since $(X,\Delta)$ is klt, $(X',\Delta')$ is klt by \cite[Proposition 5.20]{KM} and hence $X'$ has rational singularities by \cite[Theorem 5.22]{KM}.
Clearly, $\pi':=p_2'\circ p_1'^{-1}:X'\dashrightarrow A$ is a dominant rational map, since $p_1'$ is birational and $p_2'$ is surjective.
Then $\pi'$ is a surjective morphism (with $p_2'=\pi'\circ p_1'$) by Lemma \ref{alb-morphism} and the universal property of the Albanese map.
Suppose $\pi$ is not defined over some closed point $x\in X$. Then $\dim(W_x)>0$ with $W_x=p_1^{-1}(x)$ and $\dim(p_2(W_x))>0$ by \cite[Lemma 1.15]{De}.
Hence, $\dim(p_2'(\tau_2^{-1}(W_x)))>0$ and then $\dim(p_1'(\tau_2^{-1}(W_x)))>0$.
However, $p_1'(\tau_2^{-1}(W_x))=\tau_3^{-1}(x)$ has only finitely many points.
This is a contradiction.
\end{proof}

\begin{lemma}\label{polarsurj} Let $X$ be a projective variety with a polarized endomorphism $f:X\to X$.
Then the Albanese map $\alb_X:X\dashrightarrow \Alb(X)$ is a dominant rational map.
\end{lemma}

\begin{proof} Replacing $X$ by the normalization of the graph of $\alb_X$ and by Lemmas \ref{lem-lift-graph} and \ref{lem-lift-normal}, we may assume $X$ is normal and $\alb_X$ is a well-defined morphism.
Replacing $X$ by the base of the special MRC fibration of $X$ and by Lemmas \ref{smrc}, \ref{albmrc} and Theorem \ref{desends-polar}, it suffices to consider the case when $X$ is a non-uniruled normal projective variety.
By Lemma \ref{struclem}, $X$ has only canonical singularities with $K_{X}\sim_{\mathbb{Q}}0$.
In particular, $\kappa(X)=0$.  By \cite[Theorem 1]{Ka81} and Lemma \ref{alb-morphism}, $\alb_X:X\dashrightarrow\Alb(X)$ is a surjective morphism.
\end{proof}


\begin{proof}[Proof of Corollary \ref{mainthm}] (1) follows from Lemma \ref{polarsurj}.
Then (2) follows from (1) and Theorem \ref{desends-polar}.
\end{proof}

\begin{proof}[Proof of Proposition \ref{PropA}](1) follows from Lemmas  \ref{smrc} and Theorem \ref{desends-polar}; see \cite[Corollary 4.20]{Na10} for a different proof of $g$ being polarized.
(2) follows from Lemma \ref{struclem}.
(3) follows from Lemma \ref{fibres-rc+irr}.
(4) follows from Lemma \ref{mor-q-abelian}.
\end{proof}

\section{Minimal Model Program for polarized endomorphisms}\label{MMP}

We follow the approach in \cite[Lemma 2.10]{Zh-comp} and get the following general result:

\begin{lemma}\label{finite-orbit} Let $f:X\to X$  be a polarized endomorphism of a projective variety.
Suppose $A\subseteq X$ is a closed subvariety with $f^{-i}f^i(A) = A$ for
all $i\ge 0$. Then $M(A) := \{f^i(A)\,|\, i \in \Z\}$ is a finite set.
\end{lemma}

\begin{proof} We may assume $n := \dim(X) \ge 1$.
By the assumption, $f^{\ast}H\sim qH$ for some ample Cartier divisor $H$ and integer $q>1$.
Set $M_{\ge0}(A):=\{f^i(A)\,|\, i \ge 0\}$.

We first assert that if $M_{\ge 0}(A)$ is a finite set, then so is $M(A)$.
Indeed, suppose $f^{r_1}(A)=f^{r_2}(A)$ for some $0<r_1<r_2$.
Then for any $i>0$, $f^{-i}(A)=f^{-i}f^{-sr_1}f^{sr_1}(A)=f^{-i}f^{-sr_1}f^{sr_2}(A)=f^{sr_2-sr_1-i}(A)\in M_{\ge 0}(A)$ if $s\gg 1$.
So the assertion is proved.

Next we show that $M_{\ge0}(A)$ is a finite set by induction on the codimension of $A$ in $X$.
We may assume $k := \dim(A)<\dim(X)$.
Let $\Sigma$ be the union of $\Sing (X)$, $f^{-1}(\Sing (X))$ and the irreducible components in the ramification divisor $R_f$ of $f$.
Set $A_i := f^i(A) \, (i\ge 0)$.

We claim that $A_i$ is contained in $\Sigma$ for infinitely many $i$.
Otherwise, replacing $A$ by some $A_{i_0}$, we may assume that $A_i$ is not contained in $\Sigma$ for all $i\ge 0$.
So we have $f^{\ast}A_{i+1}= a_iA_i$ with $a_i\in \mathbb{Z}_{>0}$ and
$$q^nH^{k}\cdot A_{i+1}=(f^{\ast}H)^{k}\cdot f^{\ast}A_{i+1}=a_iq^{k}H^{k}\cdot A_{i},$$
$$1\leq H^{k}\cdot A_{i+1}=\frac{a_i}{q^{n-k}}\cdots \frac{a_1}{q^{n-k}}H^{k}\cdot A_1.$$
Thus for infinitely many $i$, $a_i\geq q^{n-k} > 1$. Hence $A_i\subseteq \Sigma$. This proves the claim.

If $k=n-1$, by the claim, $f^{r_1}(A)=f^{r_2}(A)$ for some $0<r_1<r_2$.
Then $|M_{\ge 0}(A)|<r_2$.

If $k\leq n-2$, assume that $|M_{\ge 0}(A)|=\infty$.
Let $B$ be the Zariski-closure of the union of those $A_{i_1}$ contained in $\Sigma$.
Then $k+1\leq \dim(B)\le n-1$, and $f^{-i}f^i(B) = B$ for all $i \ge 0$. Choose $r \ge 1$ such that $B' := f^r(B), f(B'), f^2(B'), \cdots$ all have the same number of irreducible components.
Let $X_1$ be an irreducible component of $B'$ of maximal dimension.
Then $k+1\leq \dim(X_1)\le n-1$ and $f^{-i}f^i(X_1) = X_1$ for all $i \ge 0$. By induction, $M_{\ge 0}(X_1)$ is a finite set.
So we may assume that $f^{-1}(X_1)=X_1$, after replacing $f$ by a positive power and $X_1$ by its image.
Note that $f|_{X_1}$ is polarized.
Now the codimension of $A_{i_1}$ in $X_1$ is smaller than that of $A$ in $X$.
By induction, $M_{\ge 0}(A_{i_1})$ and hence $M_{\ge 0}(A)$ are finite.

\end{proof}

Let $X$ be a log canonical (lc) normal projective variety. We refer to \cite[Theorem 1.1]{Fu11} for the cone theorem and \cite[Corollary 1.2]{Bi} for the existence of log canonical flips.

\begin{lemma}\label{equi-mmp} Let $X$  be a lc normal projective variety and $f:X\to X$ a surjective endomorphism.
Let $\pi: X\to Y$ be a contraction of a $K_X$-negative extremal ray $R_C:=\mathbb{R}_{\ge 0}[C]$ generated by some curve $C$.
Suppose that $E \subseteq X$ is a subvariety such that $\dim(\pi(E))<\dim(E)$ and $f^{-1}(E)=E$.
Then replacing $f$ by a positive power, $f(R_C)=R_C$;
hence, $\pi$ is $f$-equivariant.
\end{lemma}

\begin{proof} Since $\dim(\pi(E))<\dim(E)$, we may assume $C\subseteq E$.
Denote by $\N^1_{\mathbb{C}}(X):=\N^1(X)\otimes_\R \mathbb{C}$.
By the cone theorem (cf.~\cite[Theorem 1.1(4)iii]{Fu11}, or \cite[Corollary 3.17]{KM}), we have the linear exact sequence $$0\to \N^1_{\mathbb{C}}(Y)\xrightarrow{\pi^{\ast}}\N^1_{\mathbb{C}}(X)\xrightarrow{\cdot C} \mathbb{C}\to 0.$$
So $\pi^{\ast}\N^1_{\mathbb{C}}(Y)$ is a subspace in $\N^1_{\mathbb{C}}(X)$ of codimension $1$.
Let $i:E\hookrightarrow X$ be the inclusion map.
For any $\R$-Cartier divisor $D$ of $X$,
denote by $D|_E:=i^*D\in \N^{1}_{\mathbb{C}}(E)$ the pullback.
Let $\N^{1}_{\mathbb{C}}(X)|_E:=i^*(\N^{1}_{\mathbb{C}}(X))$ which is a subspace of $\N^{1}_{\mathbb{C}}(E)$.
Let $L := \pi^{\ast}\N^{1}_{\mathbb{C}}(Y)|_E:=i^*\pi^*(\N^{1}_{\mathbb{C}}(Y))$.
Then $L$ is a subspace in $\N^{1}_{\mathbb{C}}(X)|_E$ of codimension at most $1$.
Note that for an ample divisor $H$ in $X$, $H|_E\cdot C=H\cdot C\neq 0$.
Therefore, $H|_E\not\in L$ and hence $L$ has codimension $1$ in $\N^{1}_{\mathbb{C}}(X)|_E$.
Let $S:=\{D|_E\in \N^{1}_{\mathbb{C}}(X)|_E:(D|_E)^{\dim(E)}=0\}.$

We claim that $S$ is a hypersurface (an algebraic set defined by a non-zero polynomial) in
the complex affine space $\N^{1}_{\mathbb{C}}(X)|_E$ and $L$ is an irreducible component of $S$
in the sense of Zariski topology.
Indeed, let $\{e_1,\cdots,e_k\}$ be a fixed basis of $\N^{1}_{\mathbb{C}}(X)|_E$.
Then $S=\{(x_1,\cdots, x_k)\,|\, (\sum\limits_{i=1}^k x_ie_i)^{\dim(E)}=0\}$ is determined by a homogeneous polynomial of degree $\dim(E)$ and the coefficient of the term $\prod_i x_i^{\ell_i}$ is the intersection number $e_1^{\ell_1}\cdots e_k^{\ell_k}$.
Note that for an ample divisor $H$ in $X$, $H|_E\in \N^{1}_{\mathbb{C}}(X)|_E$ and $(H|_E)^{\dim(E)}=H^{\dim(E)}\cdot E>0$.
So $e_1^{\ell_1}\cdots e_k^{\ell_k}\neq 0$ for some $\ell_i$.
In particular, $S$ is determined by a non-zero polynomial.
Since $\dim(\pi(E))<\dim(E)$, $\pi_*E=0$.
For any $P\in \N^{1}_{\mathbb{C}}(Y)$, $(\pi^*(P)|_E)^{\dim(E)}=\pi^*(P)^{\dim(E)}\cdot E=P^{\dim(E)}\cdot (\pi_*E)=0$ by the projection formula.
So $\pi^*(P)|_E\in S$. Hence $L \subseteq S$. 
Since $L$ and $S$ have the same dimension, $L$ is an irreducible component of $S$.
The claim is proved.

The pullback $f^{\ast}$ induces an automorphism of $\N^{1}_{\mathbb{C}}(X)|_E$.
Note that $f^{\ast}E= aE$ (as cycles) for some $a>0$, and $(f^{\ast}D)^{\dim(E)}\cdot E=\frac{\deg f}{a}D^{\dim(E)}\cdot E$.
Hence, $D\in S$ if and only if $f^{\ast}D\in S$. This implies that $S$ is $f^{\ast}$-invariant.
Replacing $f$ by a positive power, $L$ is $f^{\ast}$-invariant.
In particular, for any $P\in \N^{1}_{\mathbb{C}}(Y)$, $\pi^{\ast}P|_E=(f|_E)^{\ast}(\pi^{\ast}P'|_E)=(f^\ast\pi^\ast P')|_E$ for some $P'\in \N^{1}_{\mathbb{C}}(Y)$.
By Lemma \ref{ext} and since $f^{-1}(E)=E$, we have $f^{-1}(R_C)=R_{C'}$ for some curve $C' \subseteq E$ with $f(C') = C$.
Write $f_{\ast}(C')= eC$ where $e>0$.
Now we have $$\pi^{\ast}P\cdot C'=f^\ast\pi^\ast P'\cdot C'=f_\ast(f^\ast\pi^\ast P'\cdot C')=\pi^\ast P'\cdot eC=0.$$
Thus, $R_{C'}=R_C$ and hence $f(R_C)=R_C$.
The last assertion is true since the contraction $\pi$ is uniquely determined by the ray $R_C$.
\end{proof}

\begin{remark}\label{rmk-fano} In Lemma \ref{equi-mmp}, if $E = X$ i.e., if $\pi$ is a Fano contraction,
then $\pi$ is $f^s$-equivariant for some $s>0$.
This is also a corollary of \cite[Theorem 2.2]{Wi} by showing that $X$ has only finitely many Fano contractions.
\end{remark}


\begin{lemma}\label{equi-div} Let $f:X\to X$ be a polarized endomorphism of a $\Q$-factorial lc projective variety $X$.
Suppose that $\pi:X\to X_1$ is a divisorial contraction of a $K_X$-negative extremal ray $R_C := \mathbb{R}_{\ge0}[C]$ generated by some curve $C$. Then $f^s$ descends to a surjective endomorphism of $X_1$ for some $s>0$.
\end{lemma}

\begin{proof} Let $E$ be the exceptional divisor. Then $E$ is irreducible (cf.~\cite[Proposition 2.5]{KM}). By Lemma \ref{ext}, $f^{-i}f^i(E)=E$ for all $i\ge0$. By Lemma \ref{finite-orbit}, $M(E)$ is a finite set. So we may assume $f^{-1}(E)=E$ after replacing $f$ by its positive power. Then $\pi$ is $f^s$-equivariant for some $s>0$ by Lemma \ref{equi-mmp}.
\end{proof}

\begin{lemma}\label{equi-flip} Let $f:X\to X$ be a polarized endomorphism of a lc projective variety $X$.
Let $\sigma:X\dashrightarrow X^+$ be a flip with $\pi : X\to Y$ the corresponding flipping contraction of a $K_X$-negative extremal ray $R_C := \mathbb{R}_{\ge0}[C]$ generated by some curve $C$. Then the commutative diagram
$$ \xymatrix{
  X \ar@{.>}[rr] \ar[dr]_{\pi} &  &    X^+ \ar[dl]^{\pi^+}    \\
                & Y                 }$$
is $f^s$-equivariant for some $s>0$.
\end{lemma}

\begin{proof} Let $E$ be the exceptional locus of $\pi$.
By Lemma \ref{ext}, $f^{-i}f^i(E) = E$ for all $i \ge 0$.
Choose $i_0 \ge 0$ such that $E':= f^{i_0}(E), f(E'), f^2(E'), \cdots$
all have the same number of irreducible components.
Then $f^{-i}f^i(E'(k)) = E'(k)$ for every irreducible component $E'(k)$ of $E'$.
By Lemma \ref{finite-orbit}, $M(E'(k))$ is a finite set.
Then $f^r(E'(k))=f^s(E'(k))$ for some $r>s\ge0$ and hence $f^{-i_0}(E'(k))=f^{-i_0}(f^{-r}f^{r}(E'(k)))=f^{-i_0}(f^{-r}f^{s}(E'(k)))=f^{-(r-s)}(f^{-i_0}(E'(k)))$.
So $f^{-(r-s)}$ permutes the irreducible components of $f^{-i_0}(E'(k))$.
Let $E(k)$ be an irreducible component of $E$ such that $f^{i_0}(E(k))=E'(k)$.
Then $f^{-t}(E(k))=E(k)$ for some $t>0$.
Since $\dim(\pi(E(k)))<\dim(E(k))$, we have $f^s(R_C)=R_C$ for some $s>0$ by applying Lemma \ref{equi-mmp} to $f^t$ and $E(k)$.
Hence, the rational maps on $Y$ and $X^{+}$ induced from $f^s$ are well-defined morphisms by the following Lemma \ref{nak}.
\end{proof}

The following lemma is true by using the same proof of \cite[Lemma 3.6]{Zh-comp} since log canonical flips are now known to exist (cf.~\cite[Corollary 1.2]{Bi}).
\begin{lemma}\label{nak}(cf.~\cite[Lemma 3.6]{Zh-comp}) Let $X$ be a  normal projective variety with at worst lc singularities, $f : X \to X$ a surjective endomorphism, and $X\dashrightarrow X^+$ a flip with $\pi : X \to Y$ the corresponding flipping contraction of a $K_X$-negative extremal ray $R_C := \mathbb{R}_{\ge0}[C]$ generated by some curve $C$.
Suppose that $R_{f(C)} = R_C$.
Then the dominant rational map $f^+ : X^+ \dashrightarrow X^+$ induced from $f$, is holomorphic.
Both $f$ and $f^+$ descend to one and the same endomorphism of $Y$.
\end{lemma}

\begin{definition}\label{move-def} (cf.~\cite{ZDA})
Let $X$ be a normal projective variety and $D$ an $\mathbb{R}$-Cartier divisor.
We say $D$ is \emph{movable} if: for any $\epsilon>0$, any ample divisor $H$ and any prime divisor $\Gamma$,
there is an effective $\mathbb{R}$-Cartier divisor $\Delta$ such that $\Delta\equiv D+\epsilon H$ and $\Gamma\not\subseteq \Supp\Delta$. 
\end{definition}

\begin{lemma}\label{kmovelem}Let $f:X\to X$ be a polarized endomorphism of a $\mathbb{Q}$-factorial klt projective variety $X$.
Suppose that $K_X$ is movable. Then $X$ is $Q$-abelian.
\end{lemma}

\begin{proof} First we claim that $K_X\sim_{\mathbb{Q}} 0$.
Since $K_X$ is movable, $K_X$ is pseudo-effective and satisfies the first condition of Lemma \ref{nzlem}.
Suppose that $f^\ast H\sim_{\mathbb{Q}} qH$ for some ample divisor $H$ of $X$ and $q>1$.
Taking intersection numbers with $(f^{\ast}H)^{n-1}=f^{\ast}H\cdots f^{\ast}H$ of the both sides of the ramification divisor formula $K_X=f^{\ast}K_X+R_f$, we obtain $$(q-1)K_X\cdot H^{n-1}+R_f\cdot H^{n-1}=0.$$
Since $K_X$ and $R_f$ are pseudo-effective, $K_X\cdot H^{n-1}=0$.
So by Lemma \ref{nzlem}, $K_X\equiv 0$ and hence $K_X\sim_{\mathbb{Q}} 0$ by \cite[Chapter V, Corollary 4.9]{ZDA}.

Now the lemma follows from Lemma \ref{struccor}.
\end{proof}

\begin{lemma}\label{pecase} Let $f:X\to X$ be a polarized endomorphism of a $\mathbb{Q}$-factorial klt projective variety $X$.
Assume that $K_X$ is pseudo-effective. Then $X$ is $Q$-abelian.
\end{lemma}

\begin{proof} We run MMP on $X$. Since $K_X$ is pseudo-effective, we will never arrive at a non-birational contraction.
By running finitely many steps, we get a birational map $\pi:X\dashrightarrow Y$ such that any MMP starting from $Y$ will always be a sequence of flips.
Replacing $f$ by a positive power, $f$ descends step by step by Lemmas \ref{equi-div}, \ref{equi-flip}, and Theorem \ref{desends-polar}.
Let $g=f|_Y$ and $g^\ast H\sim_{\mathbb{Q}}qH$ for some ample Cartier divisor $H$ of $Y$ and $q>1$.

We claim that $K_Y$ is movable.
Take an ample divisor $A$ of $Y$ and a sequence of positive numbers $t_j$ approaching $0$.
By \cite{BCHM}, we can run $(K_Y+t_jA)$-MMP on $Y$ with scaling of $A$ to obtain a birational map $$\pi_j: (Y,t_jA)\dashrightarrow (Y_j, t_jA_j),$$
such that $K_{Y_j}+t_jA_j$ is nef and hence movable.
Since $Y$ and $Y_j$ are isomorphic in codimension $1$ and they are both $\mathbb{Q}$-factorial,
we have a natural isomorphism $\pi_j^{\ast}: \N^1(Y_j)\to \N^1(Y)$ and $D\in \N^1(Y_j)$ is movable if and only if $\pi_j^\ast D$ is movable.
So $K_Y+t_jA=\pi_j^\ast (K_{Y_j}+t_jA_j)$ is movable for each $j$.
In particular, $K_Y$ is movable.

Note that $Y$ is $Q$-abelian by Lemma \ref{kmovelem} and $g$ is polarized. Then $X\to Y$ is a birational equi-dimensional morphism by Lemma \ref{mor-q-abelian}.
This is possible only when $X\to Y$ is an isomorphism by the Zariski's main theorem (cf.~\cite[Chapter V, Theorem 5.2]{Har}).
\end{proof}

\section{Examples of polarized endomorphisms}

In the example below, $f:Z\to Z$ is polarized, but, $(f^i)^*|_{\N^1(Z)}$ is not a scalar multiplication for any $i>0$.

\begin{example}\label{ex1} \rm{}Let $E$ be an elliptic curve with no complex multiplication and $Z=E\times E$ an abelian surface with an endomorphism $f$ corresponding to the matrix $\begin{pmatrix} 1 & -5 \\ 1 & 1 \end{pmatrix}$.
Then $\rho(Z)=3$ and $(f^i)^*|_{\N^1(Z)}$ is not a scalar multiplication for any $i>0$; see \cite[Example 4.1.5]{ENS}.
In particular, $f^\ast|_{\N^1(Z)}$ has only one real eigenvalue (counting multiplicities) and the spectral radius of $f^\ast|_{\N^1(Z)}$ is $6$;
further $f^\ast H\equiv 6H$ for some nef divisor $H\not\equiv 0$ by applying the Perron-Frobenius theorem to $\Nef(Z)$.
We claim that $H^2>0$ and hence $H$ is ample; see Remark \ref{rmk-abelian-ample}.
If $H^2=0$, then $H^{\perp}:=\{D\in \N^1(Z)\,|\,D\cdot H=0\}$ is a $f^\ast$-invariant $2$-dimensional subspace.
Note that $H\in H^{\perp}$.
So $f^{\ast}|_{\N^1(Z)}$ has two real eigenvalues, counting multiplicities.
This is a contradiction.
Hence $f$ is polarized.
\end{example}

Next we show the existence of a polarized endomorphism $g:S\to S$, such that $(g^i)^\ast|_{\N^1(S)}$ is not a scalar multiplication for any $i>0$ while
$(g|_{\Alb(S)})^\ast|_{\N^1(\Alb(S))}$ is a scalar multiplication.

\begin{example}\label{ex2} \rm{} We use the notation in Example \ref{ex1}.
Let $G$ be a group generated by $\diag[-1,-1]$ and denote by $S=Z/G$ which is a normal K3 surface.
Since $q(S)=0$ and $S$ has rational singularities, the Albanese map is trivial by Lemma \ref{alb-morphism}.
Note that $f$ is $G$-equivariant, and $\pi:Z\to S$ is a finite surjective morphism.
So $g=f|_S$ is also polarized by Theorem \ref{desends-polar}.
Next we claim that $(g^i)^\ast|_{\N^1(S)}$ is not a scalar multiplication for any $i>0$.
Clearly, it suffices to show that $\rho(S)\ge 2$ (then it follows that $\rho(S)=3$ since $f^\ast|_{\N^1(Z)}$ has only one real eigenvalue, counting multiplicities).
Suppose $\rho(S)=1$.
A fibre $E_0$ of $Z\to E$ has $E_0^2=0$.
Since $E_0$ is $G$-invariant, $\pi^\ast\pi(E_0)\equiv aE_0$ for some $a>0$.
Then $0=a^2E_0^2=(\pi^\ast\pi(E_0))^2=4\pi(E_0)^2$ and hence $\pi(E_0)$ is not ample, a contradiction.
\end{example}

\begin{example}\label{ex-xu} \rm{}We construct polarized endomorphisms $f:X\to X$ such that:
\begin{itemize}
\item[(1)]
$\dim (X) = m+n-1$ with $m\in\{4,6\}$ and $0<n<m$,
\item[(2)]
$X$ has $\mathbb{Q}$-factorial terminal (quotient) singularities and is rationally connected,
\item[(3)]
the smooth locus $X_{\reg}$ of $X$ has infinite algebraic fundamental group $\pi_1^{\alg}(X_{\reg})$,
\item[(4)]
$q^{\natural}(X) = n > 0$ (see Definition \ref{def-q}),
\item[(5)]
the Iitaka $D$-dimension satisfies $\kappa(X, -K_X) = m-1$, and
\item[(6)]
the ramification divisor $R_f \subseteq X$ of $f$ is non-trivial.
\end{itemize}

Indeed, let $G \cong \mathbb{Z}/(m)$
act on $\mathbb{P}^{m-1}$ as a (coordinates) permutation subgroup of $S_m$
so that $G$ has no non-trivial pseudo-reflections (i.e.~for any non-trivial $g\in G$, $g$ fixes at most a codimension $2$ subset), and $\mathbb{P}^{m-1}/G$ has only canonical
singularities. This is guaranteed if the age $a(h) \ge 1$ at every point
fixed by a non-trivial $h$ in $G$, e.g. if $m=4, 6$; see the proof of \cite[Lemma 3]{KX}.
Let $G = \langle g = \exp(\frac{2 \pi \sqrt{-1}}{m}) \rangle$ act diagonally on the abelian variety
$A = E^n = E \times \cdots \times E$, with $E$ being an elliptic curve
such that $G$ has no non-trivial pseudo-reflection (and hence $A/G$ is $Q$-abelian) and $A/G$ is
rationally connected. This is achievable by letting $E = \mathbb{C}/(\mathbb{Z} + \mathbb{Z }\zeta_m)$
and choosing suitable $m > n>0$ (e.g., $m=4, 6$ and $0<n<m$); see \cite{COV} and \cite[Corollary 25]{KL}.
Let $G$ act diagonally on $W = \mathbb{P}^{m-1} \times A$. Then
$X = W/G$  projects to rationally connected $A/G$ with the general fibre $\mathbb{P}^{m-1}$ and hence it is also rationally connected
by \cite[Corollary 1.3]{GHS}.  For any non-trivial $g\in G$, $g|_A$ contributes a positive value to the age $a(g)$ and hence $a(g)>1$. So $X$ is $\mathbb{Q}$-factorial terminal.
Now the multiplication map $\mu_r: A \to A$, $a \mapsto ra$, with $r \ge 2$, is polarized
such that $\mu_r^*H = r^2H$ for any symmetric ample divisor $H$ on $A$.
The power map $q_P : \mathbb{P}^{m-1} \to \mathbb{P}^{m-1}, [X_0: \cdots : X_{m-1}] \mapsto [X_0^q:\cdots : X_{m-1}^q]$
with $q = r^2$, is also polarized. Thus $f_W = (q_P, \mu_r)$ is a polarized endomorphism of $W$
and it descends to a polarized endomorphism $f$ on $X$ ($f_W$ commutes with the $G$-action).
Since $G$ also has no non-trivial pseudo-reflections on $W$, the quotient map $\gamma: W \to X$ is quasi-\'etale, $K_W = \gamma^\ast K_X$, and $\kappa(X, -K_X) = \kappa(W, -K_W) = m-1$.
Hence the topological fundamental group $\pi_1(X_{\reg})$ of the smooth locus of $X$
is the extension of $\mathbb{Z}/(m)$ by $\mathbb{Z}^{\oplus 2\dim (A)}$, and $q^{\natural}(X) = \dim (A) > 0$.
For (6), we take $D_i'$ as the pullback to $W$ of the coordinate hyperplane $\{X_i = 0\} \subseteq \mathbb{P}^{m-1}$.
Then $f_W^*D_i' = qD_i'$. Now  $R_f \ge (q-1)D_i$ with $D_i \subseteq X$ the image of $D_i'$.

\end{example}

\section{Proof of Theorem \ref{scalarthm}}

\begin{proof}[Proof of Theorem \ref{scalarthm}] If $K_X$ is pseudo-effective, then (1) follows from Lemma \ref{pecase} and (3) and (4) are trivial.
Next we consider the case where $K_X$ is not pseudo-effective.

By \cite[Corollary 1.3.3]{BCHM}, since $K_X$ is not pseudo-effective, we may run MMP with scaling for a finitely many steps: $X=X_1\dashrightarrow\cdots\dashrightarrow X_j$ (divisorial contractions and flips) and end up with a Mori's fibre space $X_j\to X_{j+1}$.
Note that $X_{j+1}$ is again $\mathbb{Q}$-factorial (cf.~\cite[Corollary 3.18]{KM} with klt singularities (cf.~\cite[Corollary 4.5]{Fu}).
So by running the same program several times, we may get the following sequence:  $$(\ast)\,\,X=X_1\dashrightarrow \cdots \dashrightarrow X_i \dashrightarrow \cdots \dashrightarrow X_r=Y,$$ such that $K_{X_{r}}$ is pseudo-effective.
Replacing $f$ by a positive power, suppose $f=f_1$ descends to a polarized endomorphism $f_{i-1}:X_{i-1}\to X_{i-1}$ via the above sequence.
By  Lemmas \ref{equi-div}, \ref{equi-flip} and Remark \ref{rmk-fano},
one can further descends $f_{i-1}$ to a surjective endomorphism $f_{i}:X_i\to X_i$ after replacing $f_{i-1}$ by a positive power (i.e. replacing $f$ by a positive power).
By Theorem \ref{desends-polar}, $f_i$ is again polarized by some ample Cartier divisor $H_i$.
So the sequence $(\ast)$ is $f^s$-equivariant for some $s>0$.
Since $K_{X_{r}}$ is pseudo-effective, $Y=X_r$ is $Q$-abelian by Lemma \ref{pecase}.

By Lemma \ref{mor-q-abelian}, the composition $X_i\dashrightarrow Y$ is a morphism for each $i$.
If $X_i\dashrightarrow X_{i+1}$ is a flip, then for the corresponding flipping contraction $X_i\to Z_i$, $(Z_i,\Delta_i)$ is klt for some effective $\mathbb{Q}$-divisor $\Delta_i$ by \cite[Corollary 4.5]{Fu}.
Hence $Z_i\dashrightarrow Y$ is also a morphism by Lemma \ref{mor-q-abelian} again.
Together, the sequence $(\ast)$ is a relative MMP over $Y$.

By Lemmas \ref{mmp-rc} and \ref{fibres-rc+irr}, $X_i\to Y$ is equi-dimensional with every fibre being (irreducible) rationally connected.
Note that $K_{X_i}$ is not pseudo-effective for any $i<r$ by (1).
Then the final map $X_{r-1}\to X_r$ is a Fano contraction.
So (2) is proved.

Via the pullback, $\N^1(X_{i+1})$ can be regarded as a subspace of $\N^1(X_i)$ and hence a subspace of $\N^1(X)$.
Then $f_i^\ast|_{\N^1(X_i)}=f^\ast|_{\N^1(X_i)}$. If $X_i\dashrightarrow X_{i+1}$ is a flip, then $\N^1(X_i)=\N^1(X_{i+1})$.
If $X_i\to X_{i+1}$ is a divisorial contraction or a Fano contraction, then $\N^1(X_{i+1})$ is a codimension one subspace of $\N^1(X_i)$ by the cone theorem and $H_i\not\in \N^1(X_{i+1})$.
So $\N^1(X_i)$ is spanned by $\N^1(X_{i+1})$ and $H_i$. Together, $\N^1(X)$ is spanned by $\N^1(Y)$ and those $\{H_i\}_{i<r}$. Clearly, if $i<r$, then $\dim(X_i)>0$.
By Proposition \ref{cone} and Lemmas \ref{basic-deg} and \ref{basic-deg2}, the eigenvalue of $H_i$ is the same $q=(\deg f_i)^{1/\dim(X_i)}$. So (3) is proved.

(4) is straightforward from (3).
\end{proof}

\section{Proof of Theorem \ref{thm-smooth-rc} and Corollary \ref{main-cor-k}}

\begin{lemma}\label{smooth-scalar} Let $f:X\to X$ be a polarized endomorphism of a $\Q$-factorial klt projective variety $X$.
Suppose either one of the following conditions (which are not automatic even for rationally connected terminal $X$; see Example \ref{ex-xu}).
\begin{itemize}
\item[(i)] $q^\natural(X)=0$ (see Definition \ref{def-q}).
\item[(ii)] The smooth locus $X_{\reg}$ of $X$ has finite algebraic fundamental group $\pi_1^{\alg}(X_{\reg})$.
\end{itemize}
Then $X$ is rationally connected and $(f^s)^\ast|_{\N^1(X)}$ is a scalar multiplication for some $s>0$.
\end{lemma}
\begin{proof}  By Theorem \ref{scalarthm}, for some $s>0$, there is an $f^s$-equivariant equi-dimensional fibration $\pi:X\to Y$ with irreducible fibres and $Y$ being a $\mathbb{Q}$-abelian variety.
Let $A\to Y$ be the finite cover \'etale in codimension one with $A$ an abelian variety.
Denote by $X'$ the normalization of $X\times_Y A$ which is irreducible since $\pi$ has irreducible fibres.
Then the projection $p_1:X'\to X$ is a finite surjective morphism \'etale in codimension one.

Note that $q(A)\le q^\natural(X')=q^\natural(X)$.
So Condition (ii) implies that $\dim(A)=0$.
Assume Condition (i) and $\dim(Y)>0$.
Replacing $A$ by its \'etale cover, we may assume $\deg p_1>|\pi_1^{\alg}(X_{\reg})|$.
Since $p_1$ is  \'etale in codimension one, we get a contradiction by the purity of branch loci.
In particular, both two conditions imply that $A=Y$ is a point and hence $X$ is rationally connected by Theorem \ref{scalarthm}.

Now $(f^s)^\ast|_{\N^1(X)}$ is a scalar multiplication by Theorem \ref{scalarthm}.
\end{proof}


\begin{proposition}\label{cy-prop} Let $f:X\to X$ be a polarized endomorphism of a $\mathbb{Q}$-factorial klt projective variety $X$ with the irregularity $q(X)=0$ (this is the case when $X$ is rationally connected).
Suppose that $f^\ast|_{\N^1(X)}=q\, \id$ for some $q>1$. Then the following are true.
\begin{itemize}
\item[(1)] If the Iitaka $D$-dimension $\kappa(X,F)=0$ for a prime divisor $F$, then $f^{-1}(F)=F$.
\item[(2)] Let $D_i$ $(1 \le i \le s)$ be all the prime divisors with $f^{-1}(D_i) = D_i$ and let $D=\sum\limits_{i=1}^s D_i$. Then $s \le \dim(X) + \rho(X)$ with $\rho(X)$ the Picard number of $X$. The equality holds true only when $K_X+D \sim_{\mathbb{Q}} 0$ and hence $X$ is of Calabi-Yau type.
\item[(3)] We have the ramification divisor $R_f=(q-1)D+\Delta$ for some effective divisor $\Delta$, such that $\Delta$ and $D$ have no common irreducible component and $\kappa(X, \Delta_j) > 0$ for every  irreducible component $\Delta_j$ of $\Delta$.
\item[(4)] We have $-(K_X + D) \sim_{\mathbb{Q}} \frac{1}{q-1} \Delta$. So either $\Delta \ne 0$ and $\kappa(X, -(K_X + D)) > 0$,
or  $K_X + D \sim_\mathbb{Q} 0$ and hence $X$ is of Calabi-Yau type.
\end{itemize}
\end{proposition}
\begin{proof} We follow the approach of \cite[Claim 3.15]{Zh-comp}.
Since $q(X)=0$, numerical equivalence implies $\mathbb{Q}$-linear equivalence.

(1) If $f^{-1}(F)\neq F$, then $f^\ast F\sim_\mathbb{Q} qF$ but $f^\ast F\neq qF$. So $\kappa(X,F)>0$.

(2) follows from \cite[Theorem 1.3]{Zh-adv}.

(3) Note that $f^\ast D_i = qD_i$ for each $i$. So $R_f=(q-1)D+\Delta$ for some effective divisor $\Delta$. Clearly, $D_i$ is not an irreducible component of $\Delta$ for each $i$. So, for each $j$, $f^{-1}(\Delta_j)\neq \Delta_j$ and hence $\kappa(X, \Delta_j) > 0$ by (1).

(4) By the ramification divisor formula, $K_X=f^\ast K_X+R_f\sim_\mathbb{Q}qK_X+(q-1)D+\Delta$. Therefore, $-(K_X+D)\sim_\mathbb{Q} \frac{1}{q-1}\Delta$. If $\Delta\neq 0$, then $\kappa(X, -(K_X + D)) > 0$ by (3). If $\Delta = 0$, then $K_X+D\sim_\mathbb{Q} 0$. Therefore, $(X,D)$ is log canonical and hence $X$ is of Calabi-Yau type; see \cite[Theorem 1.4]{BH} or \cite[Theorem 1.3]{Zh-adv}.
\end{proof}

\begin{lemma}\label{K+D-eff} Let $X$ be a rationally connected normal projective  variety and $D$ a non-uniruled prime divisor such that $K_X+D$ is $\mathbb{Q}$-Cartier. Then $K_X+D$ is pseudo-effective.
\end{lemma}
\begin{proof}
Replacing $X$ by a log resolution $X'$ of the pair $(X, D)$ and $D$ by its proper transform $D'$,
we may assume that both $X$ and $D$ are smooth, so that
the argument in \cite[Theorem 3.7]{LZ} is applicable. Suppose that $K_X+D$ is not pseudo-effective.
Then there is a dominant rational map $X\dashrightarrow Y$ such that $D$ birationally dominates $Y$. In particular, $Y$ is non-uniruled. This contradicts that $X$ is rationally connected.
\end{proof}
%
%

\begin{proof}[Proof of Theorem \ref{thm-smooth-rc}] Since $X$ is smooth, any quasi-\'etale morphism onto $X$ is \'etale by purity of branch loci.
Since $X$ is smooth and rationally connected, $X$ has no non-trivial \'etale cover; see \cite[Corollary 4.18]{De}.
In particular, $q^\natural(X)=q(X)=0$.
Then (1) follows from Lemma \ref{smooth-scalar}.
The assertions (2), (3), and (4) cases (i) and (ii) are straightforward by Proposition \ref{cy-prop}.

For (4) case (iii), by Lemma \ref{K+D-eff}, $K_X+D_1$ is pseudo-effective and hence $K_X+D=-\frac{1}{q-1}\Delta$ is pseudo-effective by Proposition \ref{cy-prop}(4).
So $\Delta=0$ and $D=D_1$. Now $K_X+D\sim_\mathbb{Q}0$ and $X$ is of Calabi-Yau type by Proposition \ref{cy-prop}(4).
\end{proof}
\begin{proof}[Proof of Corollary \ref{main-cor-k}] (1) Since $X$ is not $Q$-abelian, by Theorem \ref{scalarthm},
$K_X$ is not pseudo-effective. Further, replacing $f$ by a positive power, we may run $f$-equivariant MMP $$X=X_1\dashrightarrow \cdots \dashrightarrow X_i \dashrightarrow \cdots \dashrightarrow X_r=Y,$$ such that $X_{r-1}\to X_r$ is a Fano contraction, $Y$ is $Q$-abelian, and $f|_Y$ is polarized.
Since $(X,f)$ is primitive, $Y$ is a point and $X_{i-1}\dashrightarrow X_i$ is either a divisorial contraction or a flip for each $i<r$.
Clearly, $X_{r-1}$ is then a Fano variety of Picard number one,
so $X_{r-1}$ and hence $X$ are rationally connected; see \cite{qZh}.
By Theorem \ref{scalarthm}(4), $f^\ast|_{\N^1(X)}$ is a scalar multiplication. Thus Proposition \ref{cy-prop} is applicable.

(2) By taking pullback, we may regard $\N^1(X_i)$ as a subspace of $\N^1(X)$ for each $i$.
Let $0<i_1<i_2<\cdots<i_s<r$ ($s$ can be taken as $0$) be all the indexes such that $X_{i}\to X_{i+1}$ is a divisorial contraction if $i=i_t$ for some $1\le t\le s$.
Denote by $E_t$ the exceptional divisor of $X_{i_t}\to X_{i_t+1}$.
Note that $\N^1(X_{i_t+1})$ is a codimension $1$ subspace of $\N^1(X_{i_t})$ by the cone theorem and $-E_t$ is relative ample over $X_{i_t+1}$ by Lemma \ref{kmlem}. So $\N^1(X_{i_t})$ is spanned by $\N^1(X_{i_t+1})$ and $E_t$.
Together, $\N^1(X)$ is spanned by $\N^1(X_{r-1})$ and those $E_t$ ($E_t$ may not be a prime divisor in $\N^1(X)$).

Note that the MMP is $f^s$-equivariant for some $s>0$. So replacing $f$ by a positive power, we may assume
$f^{-1}(E_t(k))=E_t(k)$ for each irreducible component $E_t(k)$ of $E_t$.
Let $R_f$ be the ramification divisor of $f$, $R_{f_{r-1}}$ the ramification divisor of $f_{r-1}$, and $R'$ the strict transform of $R_{f_{r-1}}$ on $X$.
Clearly, $R_f\ge R'$.
If $R_{f_{r-1}}=0$, then $K_{X_{r-1}}=f^*K_{X_{r-1}}\equiv qK_{X_{r-1}}$ as we see in (1), and hence $(q-1)K_{X_{r-1}}\equiv 0$.
This is absurd because $X_{r-1}$ is a Fano variety.
So $R_{f_{r-1}}\neq 0$ and hence $R'\neq 0$.
In particular, $\N^1(X)$ is spanned by $R'$ and those $E_t$, since $X_{r-1}$ is of Picard number $1$.

Let $A$ be an ample divisor in $\N^1(X)$.
We may write $A=\sum\limits_{t=1}^s a_tE_t+bR'$ for some $a_t, b\in \mathbb{R}$.
Note that $R_f\ge R'$ and $eR_f\ge \sum\limits_{t=1}^s E_t$ for some $e\gg 1$ since $R_f\ge E_t(k)$ for each $t$ and $k$.
Then $cR_f=A+F$ for some pseudo-effective $\R$-Cartier divisor $F$ if $c\gg 1$.
So $R_f$ is big and hence $-K_X \equiv \frac{1}{q-1}R_f$ is big.
\end{proof}

\par \noindent
{\bf Acknowledgement.}

The authors
would like to thank colleagues
Tien-Cuong Dinh for suggesting the proof of Lemma \ref{R-int} which is very motivating, Chenyang Xu for suggesting the consideration of quotients of products of abelian varieties and projective spaces in Example \ref{ex-xu}, and the referee of the earlier version for very careful reading and constructive suggestions.
The second-named author would like to thank the following institutes and friends for the
very warm hospitality and support
during the preparation of the paper:
Max-Planck-Institute for Math.~at Bonn,
IHES at Paris, Bayreuth Univ.~(arranged by Fabrizio Catanese and Thomas Peternell),
Univ.~of Tokyo (arranged by Keiji Oguiso)
and Cambridge Univ.~(arranged by Caucher Birkar);
he is also supported by an ARF of NUS.


\begin{thebibliography}{99}
\bibitem{Bi}
C. Birkar, Existence of log canonical flips and a special LMMP,
Publ. Math. Inst. Hautes \'Etudes Sci. \textbf{115} (2012), 325-368.
\bibitem{BCHM}
C. Birkar, P. Cascini, C. D. Hacon and J. McKernan,
Existence of minimal models for varieties of log general type. J. Amer. Math. Soc., \textbf{23}(2):405-468, 2010.
\bibitem{Bo}
S.~Boucksom, T.~de Fernex and C.~Favre,
The volume of an isolated singularity, Duke Math. J. \textbf{161} (2012), no. 8, 1455-1520.

\bibitem{Br11} M.~Brion,
On automorphism groups of fiber bundles,
Publ. Mat. Urug. \textbf{12} (2011), 39-66.

\bibitem{BH}
A.~Broustet and A.~H\"oring,
Singularities of varieties admitting an endomorphism,
Math. Ann. \textbf{360} (2014), no. 1-2, 439-456.

\bibitem{Ca}
F.~Campana,
Connexit\'e rationnelle des vari\'et\'es de Fano, Ann. Sci. \'Ecole Norm. Sup. (4) \textbf{25} (1992), no. 5, 539-545.

\bibitem{COV}
F.~Catanese, K.~Oguiso and A.~Verra,
On the unirationality of higher dimensional Ueno-type manifolds,
Rev. Roumaine Math. Pures Appl. \textbf{60} (2015), no. 3, 337-353.

\bibitem{De}
O.~Debarre,
Higher-dimensional algebraic geometry, Universitext. Springer-Verlag, New York, 2001.

\bibitem{DHZ}
T.-C.~Dinh, F.~Hu, and D.-Q.~Zhang,
Compact K\"ahler manifolds admitting large solvable groups of automorphisms, Adv. Math. \textbf{281} (2015), 333-352.
\bibitem{Fa}
N.~Fakhruddin, Questions on self-maps of algebraic varieties, J. Ramanujan Math. Soc., \textbf{18}(2):109-122, 2003.
\bibitem{FGI}
B.~Fantechi, L.~G\"ottsche, L.~Illusie, S.~L.~Kleiman, N.~Nitsure and A.~Vistoli,
Fundamental algebraic geometry,
Grothendieck's FGA explained, Mathematical Surveys and Monographs, \textbf{123}, American Mathematical Society, Providence, RI, 2005.
\bibitem{Fu}
O.~Fujino, Applications of Kawamata's positivity theorem, Proc. Japan
Acad. Ser. A Math. Sci. \textbf{75} (1999), no. 6, 75-79.
\bibitem{Fu11}
O.~Fujino, Fundamental theorems for the log minimal model program,
Publ. Res. Inst. Math. Sci. \textbf{47} (2011), no. 3, 727-789.
\bibitem{Fu15}
O.~Fujino, Some remarks on the minimal model program for log canonical pairs,
J. Math. Sci. Univ. Tokyo \textbf{22} (2015), no. 1, 149-192.
\bibitem{FKL}
M.~Fulger, J.~Koll\'ar, B.~Lehmann,
Volume and Hilbert function of $\mathbb{R}$-divisors,
Michigan Math. J. \textbf{65} (2016), no. 2, 371-387.
\bibitem{GHS}
T.~Graber, J.~Harris and J.~Starr, Families of rationally connected varieties, J. Amer. Math.
Soc., \textbf{16}(1):57-67 (electronic), 2003.
\bibitem{GKP}
D.~Greb, S.~Kebekus and T.~Peternell,
\'Etale fundamental groups of Kawamata log terminal spaces, flat sheaves, and quotients of Abelian varieties,
Duke Math. J. \textbf{165} (2016), no. 10, 1965-2004.


\bibitem{Har}
R.~Hartshorne, Algebraic Geometry, Grad. Texts in Math. \textbf{52}, Springer, 1977.

\bibitem{Ka81}
Y.~Kawamata,
Characterization of abelian varieties,
Compos. Math. \textbf{43} (1981), no. 2, 253-276.

\bibitem{Ka}
Y.~Kawamata,
Minimal models and the Kodaira dimension of algebraic fiber spaces,
J. Reine Angew. Math. \textbf{363} (1985), 1-46.

\bibitem{Ko}
J.~Koll\'ar,
Rational curves on algebraic varieties,
Ergebnisse der Mathematik und ihrer Grenzgebiete. 3. Folge. A Series of Modern Surveys in Mathematics, \textbf{32}. Springer-Verlag, Berlin, 1996.

\bibitem{KL}
J.~Koll\'ar and M.~Larsen,
Quotients of Calabi-Yau varieties,
in : Algebra, arithmetic, and geometry, In honor of Y. I. Manin on the occasion of his 70th birthday,
Birkh\"auser Progress in Mathematics \textbf{270}, 179-211 (2009).

\bibitem{KMM}
J.~Koll\'ar, Y.~Miyaoka and S.~Mori,
Rationally connected varieties,
J. Alg. Geom. \textbf{1} (1992), 429-448.

\bibitem{KM}
J.~Koll\'ar and S.~Mori,
Birational geometry of algebraic varieties,
Cambridge Tracts in Math.,
\textbf{134} Cambridge Univ. Press, 1998.



\bibitem{KX}
J.~Koll\'ar and C.~Xu,
Fano varieties with large degree endomorphisms,
\href{http://arXiv.org/abs/0901.1692}{\tt arXiv:0901.1692}.

\bibitem{Kr}
H.~Krieger and P.~Reschke,
Cohomological conditions on endomorphisms of projective varieties,
\href{http://arXiv.org/abs/1505.07088}{\tt arxiv:1505.07088}.

\bibitem{Li}
D.~I.~Lieberman,
Compactness of the Chow scheme: applications to automorphisms
and deformations of K\"ahler manifolds,
Fonctions de plusieurs variables complexes, III
(S\'em.\ Fran\c{c}ois Norguet, 1975-1977), pp.~140-186,
Lecture Notes in Math, \textbf{670}, Springer, Berlin, 1978.

\bibitem{LZ}
S. S. Y. Lu and D.-Q. Zhang, Positivity criteria for log canonical divisors and hyperbolicity,
J. Reine Angew. Math. \textbf{726} (2017), 173-186.

\bibitem{ZDA} N.~Nakayama,
Zariski-decomposition and abundance,
MSJ Memoirs Vol.~\textbf{14}, Math.\ Soc.\ Japan, 2004.
\bibitem{ENS}
N.~Nakayama, On complex normal projective surfaces admitting non-isomorphic surjective endomorphisms,
Preprint 2 September 2008.
\bibitem{Na10}
N.~Nakayama,
Intersection sheaves over normal schemes,
J. Math. Soc. Japan \textbf{62} (2010), no. 2, 487-595.


\bibitem{Na-Zh}
N.~Nakayama and D.-Q.~Zhang,
Polarized endomorphisms of complex normal varieties,
Math. Ann. \textbf{346} (2010), no. \textbf{4}, 991-1018.

\bibitem{Re}
M.~Reid,
Projective morphisms according to Kawamata,
preprint, Univ. of Warwick, 1983.


\bibitem{Wi}
J.~ Wi\'sniewski, On contractions of extremal rays of Fano manifolds, J. Reine Angew. Math. \textbf{417} (1991), 141-157.
\bibitem{Zh-comp}
D.-Q.~Zhang, Polarized endomorphisms of uniruled varieties. Compos. Math. \textbf{146} (2010), no.
\textbf{1}, 145-168.
\bibitem{Zh-adv}
D.-Q.~Zhang, Invariant hypersurfaces of endomorphisms of projective varieties, Adv. Math., 2013, \textbf{252}(3):185-203.
\bibitem{Zh-tams}
D.-Q.~Zhang, $n$-dimensional projective varieties with the action of an abelian group of rank $n - 1$,
Trans. Amer. Math. Soc. \textbf{368} (2016), no. 12, 8849-8872.

\bibitem{qZh}
Q.~Zhang,
Rational connectedness of log $Q$-Fano varieties,
J. Reine Angew. Math. \textbf{590} (2006), 131-142.

\bibitem{yZh}
Y.~Zhang, On isolated singularities with noninvertible finite endomorphism,
Adv. Math. \textbf{308} (2017), 859-878.

\end{thebibliography}
\end{document}